\documentclass[hidelinks,onefignum,onetabnum]{siamart251216}

\usepackage{lipsum}
\usepackage{amsfonts}
\usepackage{graphicx}
\usepackage{epstopdf}
\usepackage{algorithmic}
\ifpdf
  \DeclareGraphicsExtensions{.eps,.pdf,.png,.jpg}
\else
  \DeclareGraphicsExtensions{.eps}
\fi

\newsiamremark{remark}{Remark}
\newsiamremark{hypothesis}{Hypothesis}
\crefname{hypothesis}{Hypothesis}{Hypotheses}
\newsiamthm{claim}{Claim}
\newsiamremark{fact}{Fact}
\crefname{fact}{Fact}{Facts}

\headers{Conditioning and backward errors for NEPvs}{V. P. Lithell, V. Janssens, E. Jarlebring, K. Meerbergen, W. Michiels}

\title{Conditioning and backward errors for nonlinear eigenvalue problems with eigenvector nonlinearities
}

\author{Vilhelm~P. Lithell\thanks{Department of Mathematics, KTH Royal Institute of Technology, {\tt\small\{vipl,eliasj\}@kth.se}}    \and Victor~Janssens\thanks{Department of Computer Science, KU Leuven, {\tt\small\{firstname.lastname\}@kuleuven.be}} \and Elias~Jarlebring\footnotemark[1] \and Karl~Meerbergen\footnotemark[2] \and Wim~Michiels\footnotemark[2]}

\usepackage{amsopn}
\usepackage{mathtools}
\usepackage{verbatim}
\usepackage{pgfplots}
\usetikzlibrary{external}
\usetikzlibrary{decorations.pathreplacing}
\usetikzlibrary{matrix,calc}
\usetikzlibrary{arrows}
\usetikzlibrary{decorations.markings}

\DeclareMathOperator{\rank}{rank}

\newcommand{\sign}{{\operatorname{sign}}}

\newcommand{\trans}[1]{{#1^{\mathrm T}}}

\newcommand{\ph}[1]{\phantom{#1}}

\newcommand{\trace}{{\mathrm{Tr }}}

\newcommand{\norm}[1]{\|#1\|} % vilhelm was here

\newcommand{\diag}{\operatorname{diag}}

\newcounter{todocounter}[section]

\usepackage{amsmath,amssymb,subfigure,float}
\usepackage{bbm} % vipl 
\usepackage{booktabs} % vipl: better table spacing

\usepackage{pgfplots}
\usetikzlibrary{external}
\usetikzlibrary{decorations.pathreplacing}

\usepackage{pifont}
\definecolor{myblue}{RGB}{116,173,209}
\definecolor{myred}{RGB}{244,109,67}
\definecolor{mygreen}{RGB}{102,189,99}
\definecolor{mycyan}{RGB}{5,131,135} % vj was here
\definecolor{mydarkblue}{RGB}{49,54,149}
\definecolor{mydarkred}{RGB}{165,0,38} 

\iftrue% set to iftrue if you want to regen PDF-files in gfx_tmp. 
  \usepackage{tikz}
  \usepackage{pgfplots}
  \pgfplotsset{
    compat=newest,
    tick label style={font=\scriptsize},
    label style={font=\scriptsize},
    legend style={font=\scriptsize}
  }

  \pgfplotsset{select coords between index/.style 2 args={
      x filter/.code={
          \ifnum\coordindex<#1\fi
          \ifnum\coordindex>#2\fi
      }
  }}
  \usetikzlibrary{decorations.pathreplacing}
  
  \iffalse % Set to iftrue if you want to regen fig PDF-files. 
     \usepgfplotslibrary{external} % requires a recent version of pgf package
  \else
     \renewcommand{\tikzsetnextfilename}[1]{}
  \fi
\else
  \usepackage{tikzexternal}
  \tikzsetexternalprefix{gfxtmp/}
  
\fi
\ifpdf
\hypersetup{
  pdftitle={Conditioning and backward errors for NEPvs},
  pdfauthor={V. P. Lithell, V. Janssens, E. Jarlebring, K. Meerbergen and W. Michiels}
}
\fi

\begin{document}

\maketitle

% REQUIRED
\begin{abstract}
We consider eigenvalue condition numbers and backward errors for a class of symmetric nonlinear eigenvalue problems with 
eigenvector nonlinearities.
For both of these quantities, we derive explicit and computable expressions that can be evaluated with little computational effort for a given eigenpair, assuming the matrix perturbations are measured by the spectral or Frobenius norm. We also show how symmetric perturbations can be exploited in the analysis.
By means of two numerical experiments we demonstrate that problems incorporating eigenvector nonlinearities potentially need to be treated with additional care, when compared to the linear or eigenvalue-nonlinear theory. 
\end{abstract}

% REQUIRED

% REQUIRED

\section{Introduction}\label{sec:intro}
Two of the most important concepts in numerical linear algebra are the notions of conditioning and backward error.
The condition number of a problem is essential for understanding how inaccuracies in the data affect the result of a computation. 
Backward errors are in turn crucial for gaining insight into the stability of numerical methods. 
Additionally, the backward error and condition number together provide a first order approximation to the forward error, through their product \cite{Higham:2002:ACCURACY}. 
Today, most modern high-quality software packages for numerical linear algebra include methods for computing or estimating condition numbers and backward errors. 
Of particular relevance to this work, such algorithms are for instance included in both the LAPACK \cite{anderson1999lapackusersguide} and SLEPC \cite{slepcusersmanual2025} software packages. 

Both the conditioning and backward error are well studied for a number of problems, including but not limited to; linear systems \cite{higham92linsys}, least-squares problems \cite{ming1998backwarderrbounds}\cite{sun1997leastsqbackerr}, linear and generalized eigenvalue problems \cite{higham1998structured,fraysse1998gepbackerr}, and eigenvalue-nonlinear problems \cite{tisseur2000backward}\cite[Section 2.7]{Guttel2017}.
%The aim of this work is to contribute to this growing body of work by considering conditioning and backward-error analysis for a class of problems which is yet to receive this same attention, with some exceptions, see literature discussion below. 
The aim of this work is to establish an explicit description of the conditioning and backward error in this general setting, but for a different class of eigenvalue problem.
This question has not received substantial attention, 
in spite of the fact that this problem appears in several important applications.
Specifically, we will consider the analysis of conditioning and backward errors for \textit{eigenvalue problems with eigenvector nonlinearities} (NEPv). 
That is, the problem of determining $(\lambda, v)\in\mathbb{R}\times\mathbb{R}^n\backslash\{0\}$ such that 
\begin{equation}
    A(v)v = \lambda v,\label{eq:NEPv_scale_inv}
\end{equation}
where $A:\mathbb{R}^n\rightarrow\mathbb{R}^{n\times n}$ is a function mapping vectors to symmetric matrices, i.e., $A(v)=\trans{A(v)}$ for all $v$.
We assume $A(v)$ to be continuously differentiable on $\mathbb{R}^n\backslash\{0\}$.
Moreover, we require $A(v)$ to be scaling invariant, that is, for any $\alpha\in\mathbb{R}\backslash \{0\}$ and $u\in\mathbb{R}^n$, it holds that
\begin{equation}\label{eq:scale_inv_prop}
    A(\alpha u) = A(u).
\end{equation}
While this problem will be our main focus, our setting is also applicable to the problem
%finding $(\lambda, v)\in\mathbb{R}\times\mathbb{R}^n\backslash\{0\}$ such that
%we also consider a NEPv with a normalization condition, i.e., find $(\lambda, v)\in\mathbb{R}\times\mathbb{R}^n\backslash\{0\}$ such that
\begin{equation}\label{eq:NEPv_norm}
    \hat{A}(v)v = \lambda v,\quad \norm{v}=1,
\end{equation}
where the function $\hat{A} : \mathbb{R}^n \rightarrow \mathbb{R}^{n \times n}$ is not necessarily scaling invariant.
%Notice in particular that we do not require $A(v)$ to be scaling invariant in \eqref{eq:NEPv_norm}.
%A problem of the form \eqref{eq:NEPv_norm} 
Such an equation can be transformed into a problem of the form \eqref{eq:NEPv_scale_inv} in a general way, see \cite{jarlebring2014inverseiter}, as well as \Cref{remark:eigval_cond_rescaling} in what follows.

We will assume that the function $A(v)$ in \eqref{eq:NEPv_scale_inv} can be decomposed as
\begin{equation}\label{eq:spmf_form}
    A(v) = \sum_{i = 1}^m f_i(v) A_i,  
\end{equation}
%(possibly having $m=n^2$)
meaning we will focus on eigenvector nonlinearities of this form.
This assumption is not limiting, since we can always take $m=n^2$, and in applications it is often the case that $m\ll n^2$.
Here, $f_i: \mathbb{R}^n \rightarrow \mathbb{R}$ are scaling invariant scalar functions, and $A_i = \trans{A_i} \in \mathbb{R}^{n \times n}$ are coefficient matrices for $i=1,\dots, m$.
In the context of this paper, we will consider perturbations on the coefficient matrices $A_i$.

We restrict ourselves to a symmetric real formulation in \eqref{eq:NEPv_scale_inv} for several reasons.
Firstly, most of the applications in which \eqref{eq:NEPv_scale_inv} appears are naturally formulated in this setting. 
%See applications below.
%Additionally, our results are also applicable to complex Hermitian problems by splitting real and imaginary parts, and doubling the problem size, as in \cite[Section 5]{jarlebring2014inverseiter}.
Additionally, in a complex setting, our results would require $A(v)$ to be analytic, which is much more restrictive than requiring $A(v)$ to be continuously differentiable.
The need for $A(v)$ to be analytic would also mean that we could not easily treat problems constrained by the euclidean norm \eqref{eq:NEPv_norm}, as the transformation to the scaling invariant form \eqref{eq:NEPv_scale_inv} produces a problem that is not analytic.
%, cf. \Cref{remark:eigval_cond_rescaling}.

The equation \eqref{eq:NEPv_scale_inv} defines a very large class of problems, and includes several important applications as special cases.
%, some of which we now list.
%For completeness, we also mention a selection of numerical methods for the solution of these problems.
One of the most well-studied problems of this form is the Gross-Pitaevskii equation (GPE), which is widely used in the modeling of Bose-Einstein condensation, 
see for instance \cite{bao2003numericalgpe,Henning:2025:REVIEW} and the references therein.
The nonlinearities that appear in the discretized GPE are quadratic in the elements of the eigenvector, which is due to incorporating particle-particle interactions in the formulation \cite[Section 2.2]{Henning:2025:REVIEW}.
Perhaps the most commonly used method class for the GPE are the so-called gradient-flow methods \cite{bao2004normalizedgradflow,henning2020sobolevgradflow}.
%which are also considered to be the state of the art for this problem class.
A different method class that has been applied to the GPE consists of generalizations of inverse iteration \cite{jarlebring2014inverseiter,henning2022invitconv}. 
%This method has also had its convergence characteristics thoroughly investigated in \cite{henning2022invitconv}.

Without attempting to provide a complete list of additional methods and applications, we mention that most methods for \eqref{eq:NEPv_scale_inv} are iterative in nature; for instance the \emph{self-consistent field (SCF) iteration} \cite{yang2009scfanalysis,upadhyaya2021scfconv,bai2022sharpscfestimate,lu2024locallyscf,pulay1982diisacceleration,saad2025accelerationfixedpoint}, or optimization and regularization-based methods appearing in data-science applications \cite{hein2009pspectral,hein2010pspectral,bai2018robust}.
There are also methods that target specific structures in order to linearize the problem \cite{claes2022linearizenepv,janssens2025linearizing}, or exploit these structures in order to employ tools typically used to treat eigenvalue nonlinearities \cite{jarlebring2026nepvtonep_BIT,claes2023contour}.

Finally, there has been some previous work regarding perturbation theory for eigenvector-nonlinear problems, notably the work \cite{cai2020perturbation} which deals with
%However, this work deals with a related but different problem, namely 
the \textit{subspace} formulation of the NEPv.
In this problem formulation, one is essentially searching for an invariant subspace of a nonlinear function $G:\mathbb{R}^{n\times k}\rightarrow \mathbb{R}^{n\times n}$, and it has important applications for instance in electronic structure calculations \cite{martin2020electronicstructurecalc}. 
Additionally, \cite{cai2020perturbation} assumes a particular structure in this problem,
for which a condition number, perturbation theory and error bounds are obtained.
%For this problem class, \cite{cai2020perturbation} considers a condition number and perturbation theory, and obtain bounds on the quantities they define.
In contrast, the current work focuses on the problem \eqref{eq:NEPv_scale_inv} (corresponding to $k=1$), but assumes no particular structure, and we obtain explicit computable expressions for both the condition number and backward error.
Related to this work is \cite{truhar2021relativepert} that derives existence and uniqueness results for the subspace formulation of the NEPv from the perspective of relative perturbation theory.

The remainder of this article is structured as follows. 
In \Cref{sec:prelims} we begin by providing some preliminary results which we will frequently utilize for the derivations in this paper. 
\Cref{sec:cond} is devoted to analyzing the conditioning of eigenvalues and eigenvectors of \eqref{eq:NEPv_scale_inv}, while \Cref{sec:back_err} deals with deriving computable expressions for the backward error of approximate eigenpairs.
Some numerical examples that highlight important aspects of our results and how they differ from the linear and eigenvalue-nonlinear theory are provided in \Cref{sec:sims}.
Finally, we conclude by summarizing our results and providing an outlook on some future research topics in \Cref{sec:conc_and_out}.

\textit{Notation}.
The notation we employ is mostly standard.
At several points in this work we will consider the eigenvalues and eigenvectors as functions of $\varepsilon$, the size of a perturbation to the original problem.
This will be indicated by writing $\lambda(\varepsilon)$ and $v(\varepsilon)$, respectively, as well as $\lambda'(\varepsilon)$, $v'(\varepsilon)$ for their derivatives with respect to $\varepsilon$. 
When $\varepsilon=0$ we will simply write $\lambda, v$ and $\lambda', v'$.
Throughout this article, $\norm{\cdot}$ will denote the Euclidean $2$-norm when applied to vectors. 
However, with minor modifications our results are also valid if one assumes the framework of dual norms, as is done for example in \cite{higham1998structured}.
When considering norms of matrices, we will use either the spectral norm or the Frobenius norm.
When it is important to use a certain norm in particular, this is emphasized.

\section{Preliminaries}\label{sec:prelims}

We begin our analysis by considering some preliminary results relating to \eqref{eq:NEPv_scale_inv}, which we will repeatedly exploit in the forthcoming sections.
Our analysis in this paper will frequently make use of properties of the Jacobian of the left hand side of \eqref{eq:NEPv_scale_inv}, where we define the Jacobian $J(v)$ as
\begin{align}
    J(v) \coloneq \frac{\partial}{\partial v} \left( A(v)v \right).
    \label{eq:Jacobian}
\end{align}
The following property of the scaling invariance of \eqref{eq:NEPv_scale_inv} will be used throughout the article, the proof of which can be found in \cite{jarlebring2014inverseiter}.

\begin{lemma}\label{lem:scaleinv_jac}
    Let $A:\mathbb{R}^n\rightarrow\mathbb{R}^{n\times n}$ satisfy the scaling invariance property \eqref{eq:scale_inv_prop} as in \eqref{eq:NEPv_scale_inv},
    then 
    \begin{equation}
        A(v)v = J(v)v,
    \end{equation}
    for all $v\in\mathbb{R}^n$.
    In particular, if $(\lambda, v)$ is an eigenpair of \eqref{eq:NEPv_scale_inv}, then it is also an eigenpair of $J(v)$.
\end{lemma}
% \begin{proof}
% Notice that $J(v)v$ can be interpreted as a directional derivative, so that
% \begin{align*}
%     J(v)v = \left. \frac{\partial}{\partial u} \left( A(u)u \right) \right|_{u=v} v 
%     & = \lim_{\varepsilon \rightarrow 0} \frac{A(v + \varepsilon v)(v + \varepsilon v) - A(v)v}{\varepsilon} \\
%     & = \lim_{\varepsilon \rightarrow 0} \frac{A(v)(v + \varepsilon v) - A(v)v}{\varepsilon} \\
%     & = A(v)v,
% \end{align*}
% which shows the first claim. 
% The second claim follows immediately.
% \qed
% \end{proof}

The property that an eigenpair $(\lambda, v)$ of problem \eqref{eq:NEPv_scale_inv} is also an eigenpair of $J(v)$, is one reason why the Jacobian $J(v)$ plays an important role in iterative schemes for solving \eqref{eq:NEPv_scale_inv} \cite{jarlebring2014inverseiter}\cite{jarlebring2022implicit}, as well as in the condition number and backward error results derived in this work.

Due to the scaling invariance property of \eqref{eq:NEPv_scale_inv}, any rescaled version of an eigenvector $v$ is also an eigenvector of the problem.
%which is what characterizes \eqref{eq:NEPv_scale_inv} as an eigenvalue problem. 
However, in order to derive results on perturbation theory of \eqref{eq:NEPv_scale_inv}, we need the eigenvector to be isolated.
This is something that needs to be accounted for even in the case of linear eigenvalue problems, see for instance \cite{higham1998structured}.
In order to resolve the redundancy induced by the scaling invariance, the normalization condition $\frac{1}{2}\trans{v}v = 1$ is added as an additional equation.
Alternatively, one can employ a normalization condition as in \eqref{eq:NEPv_norm}.
In the former case, the eigenvalue problem can equivalently be described using the set of nonlinear equations
\begin{align}
    \label{eq:nonlin_sys_eq}
    F(\lambda,v) \coloneq \begin{bmatrix}
        A(v)v - \lambda v \\
        1 - \frac{1}{2}\trans{v}v
    \end{bmatrix} = 0.
\end{align}
The Jacobian of this system evaluated in the eigenpair $(\lambda,v)$ is
\begin{align}
\label{eq:eigvec_cond_mat_prelims}
J_F(\lambda,v) = \begin{bmatrix}
    J(v) - \lambda I & \,\, -v \\
    -\trans{v} & \,\, 0
\end{bmatrix},
\end{align}
and must be nonsingular in order for $(\lambda,v)$ to be an isolated solution of \eqref{eq:nonlin_sys_eq}.
The following Lemma connects \eqref{eq:eigvec_cond_mat_prelims} with simple eigenvalues of $J(v)$, similar to the case of the linear eigenvalue problem.
\begin{lemma}\label{lem:nonsinglar_syst_mat}
    Let $(\lambda, v)$ be an eigenpair of \eqref{eq:NEPv_scale_inv}.
    Then the matrix \eqref{eq:eigvec_cond_mat_prelims} is nonsingular if and only if $\lambda$ is a simple eigenvalue of $J(v)$.
\end{lemma}

\begin{proof}
    We start the proof with the backward implication. Let $\lambda$ be a simple eigenvalue of $J(v)$.
    %, then $v$ is the corresponding eigenvector because of Lemma \ref{lem:scaleinv_jac}
    Suppose that \eqref{eq:eigvec_cond_mat_prelims} is singular, i.e., there exists a nonzero vector $\trans{w} = \begin{bmatrix}
        \trans{w_1} & w_2
    \end{bmatrix}$ such that
    \begin{align}
    \label{eq:Jacobian_vector_multiplication}
    J_F(\lambda,v)w = \begin{bmatrix}
        (J(v) - \lambda I)w_1 - vw_2 \\
        -\trans{v}w_1
    \end{bmatrix} = 0.
    \end{align}
    Left-multiplying the first equation with the left eigenvector $u$ of $J(v)$, corresponding to $\lambda$, gives $w_2 \, (\trans{u}v) = 0$.
    If $\trans{u}v$ equals zero, $\lambda$ is a non-simple eigenvalue which gives a contradiction, so $w_2$ must be equal to zero.
    As a result, $(J(v) - \lambda I)w_1 = 0$ and because $\lambda$ is simple, the co-rank of $J(v) - \lambda I$ equals $1$, meaning $w_1$ is a multiple of the eigenvector $v$, that is, $w_1 = \alpha v$.
    Then the second equation in \eqref{eq:Jacobian_vector_multiplication} reads $|\alpha| \, \norm{v}_2^2 = 0$, and since $v$ is nonzero we have $\alpha=0$, which contradicts with the assumption that $w$ is a nonzero vector.
    Hence \eqref{eq:eigvec_cond_mat_prelims} is nonsingular.

    Conversely, let \eqref{eq:eigvec_cond_mat_prelims} be nonsingular.
    Assume that $\lambda$ is non-simple.
    Then either the co-rank of $J(v) - \lambda I$ is greater than $1$ or $\trans{u} v = 0$ where $u$ is the left eigenvector of $J(v)$ corresponding to $\lambda$.
    In the former case, the matrix 
    \begin{align*}
        \begin{bmatrix}
            J(v) - \lambda I \\ 
            -\trans{v}
        \end{bmatrix}
    \end{align*}
    is of rank at most $n-1$, so there exists a nonzero vector $w$ in its right nullspace, making $\trans{\begin{bmatrix}
        \trans{w} & 0
    \end{bmatrix}}$ a nonzero vector in the nullspace of \eqref{eq:eigvec_cond_mat_prelims}. In the latter case, it is easy to verify that $\begin{bmatrix}
        \trans{u} & 0
    \end{bmatrix}$ lies in the left nullspace of \eqref{eq:eigvec_cond_mat_prelims}. Consequently, the Jacobian \eqref{eq:eigvec_cond_mat_prelims} is singular which gives a contradiction, and $\lambda$ must be simple. %\qed
\end{proof}

% See Theorem 3.2 in "ALAN L. ANDREW, K.-W. ERIC CHU, AND PETER LANCASTER. DERIVATIVES OF EIGENVALUES AND EIGENVECTORS OF MATRIX FUNCTIONS SIAM J. MATRIX ANAL. APPL. Vol. 14, No. 4, pp. 903-926, October 1993" for the NEPlambda.

% \begin{proof}
%     Suppose \eqref{eq:eigvec_cond_mat} is nonsingular.
%     Then the equation
%     \begin{equation}
%         (J(v)-\lambda I)u + \gamma v = 0
%     \end{equation}
%     has only the trivial solution, $(u,\gamma) = (0,0)$. 
%     In particular, we have for $\gamma=1$ that $(A-\lambda I)u + v \neq 0$ for all $u$.
%     Hence, $\lambda$ corresponds only to Jordan blocks of size 1, and is semi-simple.
%     Suppose on the contrary that $\lambda$ is semi-simple. 
%     Then for all $u$ and any $\gamma\neq 0$ we have 
%     \begin{equation}
%         (J(v)-\lambda I)u + \gamma v \neq 0.
%     \end{equation}
%     Hence, \eqref{eq:eigvec_cond_mat} has a trivial nullspace, and is nonsingular.
% \end{proof}

%For this we look at solutions of  $(J(v)-\lambda I) u_1+ u_2\ v =0$ and $v^H u_1=0$. The first equation for $u_1\neq 0$ characterizes an eigenvector ($u_2=0$) or generalized eigenvector ($u_2\neq 0$). If the multiplicity is at least two, there always exists an eigenvector or generalized eigenvector orthogonal to $v$.
%
%If we have a solution with $u_2\neq 0$ (defective case) then we must have $u^H v=0$, consistent with the results on eigenvalue sensitivity.
%The other way around, if $\lambda$ is an eigenvalue of $J(v)$ with geometric multiplicity one and $u^H v=0$, then the equation with $u_2\neq 0$ is solvable for $u_1$ resultig in a singular matrix.

As a consequence of \Cref{lem:nonsinglar_syst_mat}, simple eigenvalues of $J(v)$ with eigenvector $v$ correspond to isolated solutions of the system of nonlinear equations \eqref{eq:nonlin_sys_eq} which naturally leads to the following definition concerning simple eigenpairs of \eqref{eq:NEPv_scale_inv}, similar to that of eigenvalue-nonlinear problems, or the standard eigenvalue problem.
\begin{definition}\label{def:simple_eig}
    An eigenpair $(\lambda, v)$ of \eqref{eq:NEPv_scale_inv} is called simple if $\lambda$ is a simple eigenvalue of $J(v)$ with eigenvector $v$.
\end{definition}

In what follows, condition numbers and backward error formulas will be derived for simple eigenvalues of \eqref{eq:NEPv_scale_inv}, in the sense of \Cref{def:simple_eig}.

\section{Conditioning}\label{sec:cond}

% vipl: edited slightly to have consistent notation.

%This section derives eigenvalue and eigenvector condition numbers for simple eigenpairs of \eqref{eq:NEPv_scale_inv}.
Using the developments from the previous section, we will now establish conditioning results for eigenvalues and eigenvectors of \eqref{eq:NEPv_scale_inv}.
In order to do this, we will, for fixed matrices $E_i$, $i=1,\dots,m$, investigate perturbations $\Delta A(v)$ on the function $A(v)$ of the form
%\begin{equation}
%    \Tilde{A}(v, \varepsilon) = A(v) + \Delta A(v) = A(v) + \sum_{i=1}^m f_i(v) \Delta A_i,
%\end{equation}
\begin{equation}\label{eq:perturbed_A(v)}
    A(v) + \Delta A(v, \varepsilon) := A(v) + \sum_{i=1}^m f_i(v)\varepsilon E_i,
\end{equation}
where $\norm{E_i}\leq w_i$, with $w_i \geq 0$ being prescribed weights that can be used to change how we measure the perturbations.
Intuitively, one can think of $\varepsilon$ as controlling the size of the perturbation, while $E_i$, $i=1,\dots,m$, determines the direction of the perturbation.
We will consider both arbitrary real perturbations, as well as symmetric real perturbations that preserve the structure of the original problem.
We are interested in studying how the eigenvalues and eigenvectors change under these perturbations, and in order to do this, the perturbed problem is cast into a set of nonlinear equations using a normalization condition on the eigenvectors, similar to \Cref{sec:prelims}.
The perturbed problem can then be represented as the system of nonlinear equations
%This is represented by the nonlinear function  $\tilde{F}:\mathbb{R}^{n+2} \rightarrow \mathbb{R}^{n+1}$ given by
\begin{equation}\label{eq:implicit_functions_syst}
    \tilde{F}(\lambda,v, \varepsilon) := 
    \begin{bmatrix}
        (A(v) + \Delta A(v, \varepsilon))v - \lambda v \\
        1 - \frac{1}{2} \trans{v}v
    \end{bmatrix} = 0,
\end{equation}
with $\tilde{F}(\lambda,v,0) = F(\lambda,v)$.
% Notice that $\tilde{F}(\lambda,v,0) = F(\lambda,v)$.
In what follows, we will view the eigenvalues and eigenvectors as functions of $\varepsilon$, %, the size of the perturbation. %, i.e., functions of $\varepsilon$.
by invoking the implicit function theorem.
% In order to do this, we will invoke the implicit function theorem.
Assume therefore that the matrices $E_i$ in \eqref{eq:perturbed_A(v)} are fixed, and that $\tilde{F}(\lambda, v, 0) = 0$, i.e., the nominal $(\lambda, v)$ is a solution to \eqref{eq:NEPv_scale_inv}.
Then from the implicit function theorem \cite[Theorem 9.28]{rudin1986mathematical_analysis} 
we have the following statement.
If $\tilde{F}$ is continuously differentiable
in a neighborhood of $(v, \lambda, 0)$ and $J_F(\lambda,v)$ \eqref{eq:nonlin_sys_eq} is nonsingular (i.e., $(\lambda,v)$ is simple), then there exists an open set $\mathcal{U}\subset\mathbb{R}$ containing $\varepsilon = 0$ and unique differentiable functions $v(\varepsilon)$ and $\lambda(\varepsilon)$ on $\mathcal{U}$ such that $v(0) = v$, $\lambda(0) = \lambda$.
Additionally, we have that 
\begin{equation}\label{eq:perturbed_problem_constraint}
     (A(v(\varepsilon)) + \Delta A(v(\varepsilon), \varepsilon))v(\varepsilon) - \lambda(\varepsilon) v(\varepsilon) = 0, \quad \text{and}\quad \frac{1}{2} \trans{v(\varepsilon)}v(\varepsilon) = 1,
\end{equation}
for all $\varepsilon \in \mathcal{U}$.
Furthermore, the derivatives of these functions with respect to $\varepsilon$, evaluated in $\varepsilon = 0$, satisfy
\begin{equation}
\label{eq:derivatives_of_v_lambda}
    J_F(\lambda,v)
    \begin{bmatrix}
        v' \\ \lambda'    
    \end{bmatrix}
    =
    \begin{bmatrix}
        -\sum_{i=1}^m f_i(v) E_iv \\ 0
    \end{bmatrix},
\end{equation}
where $J_F(\lambda,v)$ is given by \eqref{eq:eigvec_cond_mat_prelims}.
%These functions are used to define the eigenvalue and eigenvector condition numbers in this section.
These functions $\lambda(\varepsilon)$ and $v(\varepsilon)$ 
%defined in this manner 
will in turn be used to define the condition numbers in this section.
In particular, the condition number can be thought of as the worst case derivatives of $\lambda(\varepsilon)$ and $v(\varepsilon)$, when evaluated in $\varepsilon=0$,
that is, we consider the perturbations $E_i$, $i=1,\dots,m$, that maximize $|\lambda'|$ and $\norm{v'}$, respectively.
% Hence, for the forthcoming developments, we will often maximize quantities such as $|\lambda'|/|\lambda|$, by which we mean a maximization problem over all feasible perturbation directions $E_i$.
We refrain from writing $\lambda, v$ and $\lambda', v'$ explicitly as functions of $E_i$, for ease of notation.
%The reader should however keep this in mind when considering the upcoming definitions.

\begin{remark}[Normwise and componentwise analysis]
    In this work, only a normwise analysis is considered, i.e., we only study the effects of perturbations that are bounded in norm.
    This is however not the only formulation possible, and it is common to also consider perturbations that are bounded \textit{componentwise} (see, e.g., \cite{higham1998structured}).
    %While this is not something we consider here, we expect that our results generalize straight-forwardly to this context.
\end{remark}

\subsection{Eigenvalue condition number}

Having established the existence of functions $(\lambda(\varepsilon), v(\varepsilon))$, we now use the discussion above to study conditioning aspects of \eqref{eq:NEPv_scale_inv}.
We begin by defining the condition number of a simple eigenvalue, and subsequently derive a computable expression for this quantity.
For the perturbations we also include weights $w_i \geq 0$, $i=1,\dots,m$, to allow flexibility on the size of the perturbations.
%This notation will be used throughout.
%The discussion in the previous section motivates the following definition.

\begin{definition}\label{def:eigval_cond_def}
    Assume $(\lambda, v)$ is a simple eigenpair of \eqref{eq:NEPv_scale_inv}, and let $(\lambda(\varepsilon), v(\varepsilon))$ be the unique functions satisfying \eqref{eq:perturbed_problem_constraint}.
    Given weights $w_i \geq 0$, $i = 1, \dots, m$, the normwise eigenvalue condition number is defined as
    \begin{equation}\label{eq:eigval_cond_def}
    \begin{split}
        \kappa(\lambda) = \sup \left\{ \frac{|\lambda'|}{|\lambda|}\::\: \norm{E_i} \leq w_i, \: i = 1, \dots, m \right\},
    \end{split}
    \end{equation}
    where the supremum is taken over all feasible $E_i$.

    %Let $(\lambda, v)$ be an eigenpair of \eqref{eq:NEPv_scale_inv}, and assume $\lambda$ is simple.
    
    %%%%% Old Definition: %%%%%%%%
    % Let $(\lambda, v)$ be a simple eigenpair of \eqref{eq:NEPv_scale_inv}.
    % Then the normwise condition number of $\lambda$, $\kappa(\lambda)$, is defined to be
    % \begin{equation}\label{eq:eigval_cond_def}
    % \begin{split}
    %     \kappa(\lambda) = \limsup_{\varepsilon\rightarrow 0 }\left\{ \frac{|\lambda'(\varepsilon)|}{|\lambda|}\::\: \left(A(v(\varepsilon)) + \Delta A(v(\varepsilon))\right)v(\varepsilon) = \lambda(\varepsilon)v(\varepsilon), \right. \\
    %     \left. \phantom{\frac{|\Delta \lambda|}{\varepsilon |\lambda|}} \frac{1}{2}\trans{v(\varepsilon)}v(\varepsilon) = 1, \quad \norm{\Delta A_i}\leq \varepsilon w_i \right\},
    % \end{split}
    % \end{equation}
    % where $w_i>0$, $i = 1,\dots,m$, are prescribed weights.
    %%%%%%%%%%%%%%%%%%%%%%%%%%%%%%%%
\end{definition}

\begin{theorem}\label{thm:eigval_cond_thm}
    Let $(\lambda,v)$ be a simple eigenpair of \eqref{eq:NEPv_scale_inv}.
    Let $u$ be the left eigenvector of $J(v)$ corresponding to $\lambda$.
    Then $\kappa(\lambda)$ is given by the expression
    \begin{equation}
        \kappa(\lambda) = \frac{\left(\sum_{i=1}^m|f_i(v)|w_i\right)\norm{u}\norm{v}}{|\lambda||\trans{u}v|}.
        \label{eq:condition_number}
    \end{equation}
\end{theorem}
\begin{proof}
    The first equation in \eqref{eq:derivatives_of_v_lambda} reads
    \begin{equation}
        (J(v) - \lambda I)v' - \lambda'v =  -\sum_{i=1}^mf_i(v)E_iv.
    \end{equation}
    Multiplying from the left with $\trans{u}$, and using \Cref{lem:scaleinv_jac} yields
    \begin{equation}\label{eq:lambda_deriv_expression}
        \lambda' = \frac{\sum_{i=1}^mf_i(v)\trans{u}E_iv}{\trans{u}v}.
    \end{equation}
    Taking the absolute value of both sides, using Cauchy-Schwarz, and dividing by $|\lambda|$, gives the upper bound 
    \begin{equation}\label{eq:condition_num_upper_bound}
        \kappa(\lambda) \leq \frac{\left(\sum_{i=1}^m|f_i(v)|w_i\right)\norm{u}\norm{v}}{|\lambda||\trans{u}v|}.
    \end{equation}
    It remains to be shown that this bound can be attained by a specific choice of $E_i$.
    Let $E_i$ be defined by
    \begin{equation}
        E_i = \frac{\sign(f_i(v)) w_i u\trans{v}}{\norm{u}\norm{v}},\quad i=1,\dots,m.
    \end{equation}
    Then the norm inequalities $\norm{E_i} \leq w_i$ are satisfied as equalities for both the spectral and Frobenius norm, and we have 
    \begin{equation}
        \frac{|\lambda'|}{|\lambda|} = \frac{\left(\sum_{i=1}^m|f_i(v)|w_i\right)\norm{u}\norm{v}}{|\lambda||\trans{u}v|},
    \end{equation}
    which concludes the proof. %\qed
\end{proof}

\begin{remark}[Transformation to scaling invariant form]\label{remark:eigval_cond_rescaling}
    For the derivations in this section, we assume that the problem is scaling invariant, i.e, we consider problems of the form \eqref{eq:NEPv_scale_inv}.
    However, our theory can also treat problems that impose a normalization condition, that is, problems of the form \eqref{eq:NEPv_norm}, where the function $\hat{A}(v) = \sum_{i=1}^m \hat{f}_i(v)\hat{A}_i$ is not necessarily scaling invariant.
    It is straight-forward to show that if we perform a rescaling 
    \begin{equation}
        A(v) := \hat{A}\left(\frac{v}{\norm{v}}\right),
    \end{equation}
    and compute the corresponding Jacobian of the rescaled problem, $J(v) := \frac{\partial}{\partial v}A(v)v$, then the rescaled version of the problem satisfies the scaling-invariance property
    \begin{equation}\label{eq:jacobian_rescaling}
        J(v)v = A(v)v.
    \end{equation}
    Hence, the above results are directly applicable to the rescaled problem, and the corresponding condition number becomes
    \begin{equation}
        \kappa(\lambda) = \frac{\left(\sum_{i=1}^m|f_i(v)|w_i\right)\norm{u}\norm{v}}{|\lambda||\trans{u}v|},
    \end{equation}
    where $f_i(v)$ are the rescaled versions of $\hat{f}_i(v)$, i.e., $f_i(v) = \hat{f}_i(v/\norm{v})$.
    Similar expressions also hold for the remaining results in this paper. 
    Hence, without loss of generality, we only treat problems of the form \eqref{eq:NEPv_scale_inv} in the remaining parts of this article.
\end{remark}

Note that \Cref{def:eigval_cond_def} enforces no symmetric restrictions on the perturbations, while
%In particular, the perturbed problem is not necessarily an instance of \eqref{eq:NEPv_scale_inv}, i.e., it is not necessarily symmetric.
%Consequently, Theorem~\ref{thm:eigval_cond_thm} does not respect the symmetric structure of the problem. 
it is natural to study perturbations that conserve properties of the problem, and in order to study the behavior of simple eigenpairs of \eqref{eq:NEPv_scale_inv} under symmetric perturbations, we now define a symmetric condition number.
Additionally, we subsequently show that the expression for the eigenvalue condition number \eqref{eq:condition_number} remains unchanged when requiring the perturbations to be symmetric, if one measures the perturbations using the spectral norm.
Similar results also hold for linear and eigenvalue-nonlinear problems \cite{higham1998structured}\cite{tisseur2000backward}.

\begin{definition}\label{def:sym_eigval_cond_def}
    Assume $(\lambda, v)$ is a simple eigenpair of \eqref{eq:NEPv_scale_inv}, and let $(\lambda(\varepsilon), v(\varepsilon))$ be the unique functions satisfying \eqref{eq:perturbed_problem_constraint}.
    Given the weights $w_i \geq 0$, $i = 1, \dots, m$, the normwise symmetric eigenvalue condition number is defined as
    \begin{equation}\label{eq:sym_eigval_cond_def}
        \kappa_{\text{sym}}^{(*)}(\lambda) = \sup \left\{ \frac{|\lambda'|}{|\lambda|}\::\: \norm{E_i}_* \leq w_i, \quad E_i = \trans{E_i}, \quad i = 1, \dots, m \right\},
    \end{equation}
    where the supremum is taken over all feasible $E_i$, and where $*$ is either $2$ or $F$ for the two norm and Frobenius norm, respectively.

    %%%%%%% OLD Definition: %%%%%%%%%%%%
    % Let $(\lambda, v)$ be a simple eigenpair of \eqref{eq:NEPv_scale_inv}.
    % Then the normwise symmetric condition number of $\lambda$, $\kappa_{\text{sym}}(\lambda)$, is defined to be
    % \begin{equation}\label{eq:sym_eigval_cond_def}
    % \begin{split}
    %     \kappa_{\text{sym}}(\lambda) = \limsup_{\varepsilon\rightarrow 0 }\left\{ \frac{|\lambda'(\varepsilon)|}{|\lambda|}\::\: \left(A(v(\varepsilon)) + \Delta A(v(\varepsilon))\right)v(\varepsilon) = \lambda(\varepsilon)v(\varepsilon), \right. \\
    %     \left. \phantom{\frac{|\Delta \lambda|}{\varepsilon |\lambda|}} \frac{1}{2}\trans{v(\varepsilon)}v(\varepsilon) = 1, \quad \norm{\Delta A_i}\leq \varepsilon w_i, \quad \trans{(\Delta A_i)} = \Delta A_i \right\},
    % \end{split}
    % \end{equation}
    % where $w_i>0$, $i = 1,\dots,m$, are prescribed weights.
    %%%%%%%%%%%%%%%%%%%%%%%%%%%%%%%%
\end{definition}

Other than the spectral norm, the most common norm to employ when measuring the size of perturbations to a problem is the Frobenius norm. 
For this norm, we can still achieve the same formula \eqref{eq:condition_number} for the condition number in the case where we allow arbitrary real perturbations, by again considering an optimal perturbation where $E_i$ is proportional to $u\trans{v}$, cf. the proof of \Cref{thm:eigval_cond_thm}.
However, for symmetric perturbations, it turns out that when using the Frobenius norm we have the relation
\begin{equation}
    \kappa_{\text{sym}}^{(F)}(\lambda) \leq \kappa(\lambda),
\end{equation}
where equality holds if and only if $u = cv$ for some scalar $c$.
See also \Cref{tab:frobenius}.
We now provide formulas for the symmetric condition number \eqref{eq:sym_eigval_cond_def} when employing both the spectral and Frobenius norm.

\begin{theorem}\label{thm:sym_eigval_cond_thm}
    Let $(\lambda,v)$ be as in \Cref{def:sym_eigval_cond_def} and let $\theta$ be the angle between $u$ and $v$, i.e., $\theta := \arccos((\trans{u}v)/(\norm{u}\norm{v}))$, then
    \begin{equation}
        \kappa_{\text{sym}}^{(2)}(\lambda) = \kappa(\lambda), \quad \text{and} \quad  \kappa_{\text{sym}}^{(F)}(\lambda) = \beta \kappa(\lambda),
    \end{equation}
    where
    \begin{equation}\label{eq:beta_factor}
        \beta := \sqrt{\frac{1 + (\cos(\theta))^2}{2}}.
    \end{equation}
\end{theorem}
\begin{proof}
    For the spectral norm condition number $\kappa_{\text{sym}}^{(2)}(\lambda)$, we still have the upper bound \eqref{eq:condition_num_upper_bound}, but we need to show that this bound can be attained by a symmetric perturbation.
    Let $\hat{v}, \hat{u}$ be the normalized versions of $v$ and $u$, respectively.
    That is, we set $\hat{v} = v/\norm{v}$, and $\hat{u} = u/\norm{u}$.
    Let $H$ be the Householder reflector defined by $\hat{u}-\hat{v}$, i.e., 
    \begin{equation}
        H = I-2\frac{(\hat{u}-\hat{v})\trans{(\hat{u}-\hat{v})}}{\trans{(\hat{u}-\hat{v})}(\hat{u}-\hat{v})},
    \end{equation}
    and consider the matrices
    \begin{equation}
        E_i = \sign(f_i(v)) w_i H.
    \end{equation}
    Obviously, $E_i$ is symmetric for $i=1,\dots,m$.
    Since $H$ is symmetric and orthogonal, we have $\norm{H}_2 = 1$, meaning the inequalities in \eqref{eq:sym_eigval_cond_def} are satisfied as equalities.
    Furthermore, we have $\trans{u}Hv = \norm{u}\norm{v}$, which follows from standard properties of Householder reflectors. 
    Hence, for this choice of $H$ we have
    \begin{equation}
        |\lambda'| = \frac{\left(\sum_{i=1}^m|f_i(v)|w_i\right)\norm{u}\norm{v}}{|\trans{u}v|}.
    \end{equation}
    Dividing by $|\lambda|$ gives the stated result.

    In case the Frobenius norm is used to bound the perturbations, we start from the expression \eqref{eq:lambda_deriv_expression}, and use the fact that $E_i$ is symmetric for $i=1,\dots,m$, to obtain
    \begin{equation}
        |\trans{u}v||\lambda'| \leq \sum_{i=1}^m|f_i(v)||\trans{u}E_iv| = \sum_{i=1}^m|f_i(v)||\trace(\trans{u}E_iv)| = \sum_{i=1}^m|f_i(v)||\trace(\trans{E_i}(v\trans{u}))|,
    \end{equation}
    where we used the cyclic property of the trace operator $\trace()$.
    Again, since the $E_i$ are symmetric, we can write
    \begin{equation}\label{eq:trace_trick}
        \sum_{i=1}^m|f_i(v)||\trace(\trans{E_i}(v\trans{u}))| = \sum_{i=1}^m|f_i(v)|\left|\trace\left(\trans{E_i}\left(\frac{v\trans{u} + u\trans{v}}{2}\right)\right)\right|.
    \end{equation}
    Using Cauchy-Schwarz with the Frobenius inner product then gives us the bound
    \begin{equation}\label{eq:rank_2_symm_matrix_frob_norm}
        |\trans{u}v||\lambda'| \leq \sum_{i=1}^m|f_i(v)|w_i\left\|\frac{v\trans{u} + u\trans{v}}{2}\right\|_F.
    \end{equation}
    It is straight forward to verify that $\norm{(v\trans{u} + u\trans{v})/2}_F = \norm{u}\norm{v}\beta$, so dividing both sides by $|\lambda||\trans{u}v|$ shows that we have the upper bound
    \begin{equation}
        \kappa_{\text{sym}}(\lambda) \leq \frac{\left(\sum_{i=1}^m|f_i(v)|w_i\right)\norm{u}\norm{v}\beta}{|\lambda||\trans{u}v|}.
    \end{equation}
    It remains to be shown that this bound can be attained by symmetric perturbations.
    Let
    \begin{equation}
        E_i = \frac{\sign(f_i(v))w_i}{2\beta\norm{u}\norm{v}}\left(u\trans{v} + v\trans{u}\right), \qquad i=1,\dots,m,
    \end{equation}
    then the norm inequalities are satisfied as equalities and we have 
    \begin{equation}
        \frac{|\lambda'|}{|\lambda|} = \frac{(\sum_{i=1}^m|f_i(v)|w_i)\trans{u}(u\trans{v} + v\trans{u})v}{|\lambda||\trans{u}v|2\beta\norm{u}\norm{v}} 
       % = \frac{\sum_{i=1}^m|f_i(v)|w_i\norm{u}^2\norm{v}^2 2\beta}{|\lambda||\trans{u}v|2\beta\norm{u}\norm{v}} \\
         = \frac{\left(\sum_{i=1}^m|f_i(v)|w_i\right)\norm{u}\norm{v}\beta}{|\lambda||\trans{u}v|},
    \end{equation}
    which concludes the proof.
    %\qed
\end{proof}

% vipl: no longer relevant comment, removed.
%Results similar to Theorem~\ref{thm:sym_eigval_cond_thm} can be derived for several of the other quantities in this work, in a similar manner to how this was done above. 
%For the sake of brevity, and to avoid unnecessary repetition, we will refrain from stating such results explicitly, and the reader can instead refer back to Table~\ref{tab:frobenius} for the appropriate scaling.
%%%%%%%%%%%%%%%%%%%%%%%%%%%%%%%%%%%%%%%%
%%%%%%%%%%%%%%%%%%%%%%%%%%%%%%%%%%%%%%%%
%%%%%%%%%%%%%%%%%%%%%%%%%%%%%%%%%%%%%%%%

% vipl: trying a table for frobenius stuff, might be overkill for what we want to say
\begin{table}[]
\centering
\caption{Condition number, scaled by $\hat{\alpha}:=~(\sum_{i=1}^m|f_i(v)|w_i\norm{v}\norm{u})/(|\lambda||\trans{u}v|)$, for arbitrary and symmetric perturbations, when measured with the spectral norm, and the Frobenius norm.
Here, $\theta$ is the angle between $u$ and $v$.}
\label{tab:frobenius}
\begin{tabular}{ccc}
\toprule
 & $ \kappa(\lambda)/\hat{\alpha}$ & $\kappa_{\text{sym}}(\lambda)/\hat{\alpha}$ \\ \midrule
$\norm{\cdot}_2$ & 1 & 1 \\
$\norm{\cdot}_F$ & 1 & $\sqrt{\frac{1+(\cos(\theta))^2}{2}}$ \\ \bottomrule
\end{tabular}
\end{table}

\subsection{Eigenvector condition number}
We now proceed to study conditioning of the eigenvectors of \eqref{eq:NEPv_scale_inv}.
The eigenvector condition number is defined as follows.

\begin{definition}\label{def:eigvec_cond_definition}

    Assume $(\lambda, v)$ is a simple eigenpair of \eqref{eq:NEPv_scale_inv}, and let $(\lambda(\varepsilon), v(\varepsilon))$ be the unique functions satisfying \eqref{eq:perturbed_problem_constraint}.
    Given the weights $w_i \geq 0$, $i = 1, \dots, m$, the normwise eigenvector condition number is defined as
    \begin{equation}\label{eq:eigvec_cond_def}
        \kappa(v) = \sup \left\{ \frac{\norm{v'}}{\norm{v}}\::\: \norm{E_i} \leq w_i, \quad i = 1, \dots, m \right\},
    \end{equation}
    where the supremum is taken over all feasible $E_i$.

    %%%%%%%%%%%% OLD Definition: %%%%%%%%%%%%%%%
    % Let $(\lambda, v)$ be a simple eigenpair of \eqref{eq:NEPv_scale_inv}.
    % Then the normwise condition number of $v$, $\kappa(v)$, is defined to be
    % \begin{equation}\label{eq:eigvec_cond_def}
    % \begin{split}
    %     \kappa(v) = \limsup_{\varepsilon\rightarrow 0 }\left\{ \frac{\norm{v'(\varepsilon)}}{\norm{v}}\::\: \left(A(v(\varepsilon)) + \Delta A(v(\varepsilon))\right)v(\varepsilon) = \lambda(\varepsilon)v(\varepsilon), \right. \\
    %     \left. \phantom{\frac{|\Delta \lambda|}{\varepsilon |\lambda|}} \frac{1}{2}\trans{v(\varepsilon)}v(\varepsilon) = 1, \quad \norm{\Delta A_i}\leq \varepsilon w_i  \right\},
    % \end{split}
    % \end{equation}
    % where $w_i>0$, $i = 1,\dots,m$, are prescribed weights.
    %%%%%%%%%%%%%%%%%%%%%%%%%%%%%%%%%%%%%%%%%%%%%%
\end{definition}

Similar to the eigenvalue condition number, \Cref{def:eigvec_cond_definition} does not directly provide a computable quantity for the eigenvector condition number $\kappa(v)$. To obtain this, we first derive a formula for the eigenvector sensitivity $v'$ using the following Lemmas.
% In order to obtain such an expression, we will derive an expression for the eigenvector sensitivity, $v'$.
% This will subsequently be used to derive a computable expression for \eqref{eq:eigvec_cond_def}.
% We begin with a technical lemma.

\begin{lemma}
    Let $(\lambda, v)$ be a simple eigenpair of \eqref{eq:NEPv_scale_inv}.
    If $V \in \mathbb{R}^{n \times (n-1)}$ is a full rank matrix whose columns are orthogonal to $v$, then $\trans{V}(J(v) - \lambda I)V$ is nonsingular.
\label{lem:invertibility}
\end{lemma}
\begin{proof}
    Assume $\trans{V}(J(v) - \lambda I)V$ is singular, then there exists a nonzero vector $y$ such that $\trans{V}(J(v) - \lambda I)Vy = 0$ which implies that $(J(v) - \lambda I)Vy = \alpha v$ for some non-zero scalar $\alpha \in \mathbb{R}$.
    Multiplying from the left with the left eigenvector $\trans{u}$ of $J(v)$ corresponding to the eigenvalue $\lambda$ gives $\alpha \trans{u}v = 0$, so we must have $\alpha = 0$
    %$\alpha$ must equal zero 
    since $\lambda$ is simple.
    %But then the rank of $J(v) - \lambda I$ is smaller than $n-1$
    But then $\rank(J(v) - \lambda I) < n-1$
    since the two linearly independent vectors $v$ and $Vy$ lie in its right null space which gives a contradiction, meaning $\trans{V}(J(v) - \lambda I)V$ is nonsingular.
\end{proof}

\begin{lemma}\label{lem:delta_v_lemma}
    %Let $(\lambda,v)$ be an eigenpair of \eqref{eq:NEPv_scale_inv}, and suppose $\lambda$ is simple.
    Assume $(\lambda, v)$ is a simple eigenpair of \eqref{eq:NEPv_scale_inv}, and let $(\lambda(\varepsilon), v(\varepsilon))$ be the unique functions satisfying \eqref{eq:perturbed_problem_constraint}.
    Furthermore, let $u$ be the left eigenvector of $J(v)$ corresponding to $\lambda$, and $V$ a full rank matrix whose columns are orthogonal to $v$.
    Then $v'$ in \Cref{def:eigvec_cond_definition} can be expressed as 
    \begin{equation}\label{eq:deltav_expression}
        v' = - V\left(\trans{V}(J(v)-\lambda I)V\right)^{-1}\trans{V}\left(\sum_{i=1}^mf_i(v)E_iv\right)
    \end{equation}
\end{lemma}

\begin{proof}
    Expanding the first equation of \eqref{eq:derivatives_of_v_lambda} and inserting the formula for $\lambda'$ \eqref{eq:lambda_deriv_expression}, gives 
    \begin{equation}\label{eq:delta_v_lemma_J_eq}
        (J(v)-\lambda I)v' = \left(\frac{\sum_{i=1}^mf_i(v)\trans{u}E_iv}{\trans{u}v}\right)v - \sum_{i=1}^mf_i(v)E_i v.
    \end{equation}
    The second equation in \eqref{eq:derivatives_of_v_lambda} implies that $v' = Vy$ for some $y\in\mathbb{R}^{n-1}$.
    Writing $v'$ this way, and premultiplying \eqref{eq:delta_v_lemma_J_eq} by $\trans{V}$ yields
    \begin{equation}
        \trans{V}(J(v)-\lambda I)Vy = -\trans{V}\left(\sum_{i=1}^mf_i(v)E_iv\right).
    \end{equation}
    Since $\trans{V}(J(v)-\lambda I)V$ is nonsingular by \Cref{lem:invertibility}, this means we can write $v'$ finally as 
    \begin{equation}
        v' = -V\left(\trans{V}(J(v)-\lambda I)V\right)^{-1}\trans{V}\left(\sum_{i=1}^mf_i(v)E_iv\right),
    \end{equation}
    which concludes the proof.
    % \qed
\end{proof}

% vipl: commented out, not really relevant, due to sup in def
%\begin{remark}
%Even though the matrix $\trans{V}(J(v)-\lambda I)V$ is nonsingular in the expression for $v'$, the matrix $V(\trans{V}(J(v)-\lambda I)V)^{-1}\trans{V}$ is not and has a null space spanned by the vector $v$.
%Consequently, $v'$ equals zero if and only if $\sum_{i=1}^mf_i(v)E_iv = \alpha v$ for some scalar $\alpha$.
%\end{remark}

Finally, we now derive a computable expression for the condition number $\kappa(v)$.
\begin{theorem}
    Let the matrix $V\in\mathbb{R}^{n\times(n-1)}$ be as in \Cref{lem:delta_v_lemma}. 
    Then the eigenvector condition number \eqref{eq:eigvec_cond_def} is given by
    \begin{equation}\label{eq:eigvec_cond_closed_form}
    \kappa(v) = \left\| V(\trans{V}(J(v)-\lambda I)V)^{-1}\trans{V} \right\|_2\left(\sum_{i=1}^m|f_i(v)|w_i\right).
    \end{equation}
\end{theorem}
\begin{proof}
    Using the formula for $v'$ from \Cref{lem:delta_v_lemma} directly shows that \eqref{eq:eigvec_cond_closed_form} is an upper bound for the condition number.
    Let $Z = V(\trans{V}(J(v)-\lambda I)V)^{-1}\trans{V}$.
    We show that the upper bound is attained by considering the specific perturbations %$\Delta A_i = \varepsilon E_i$, where
    \begin{equation}
        E_i = \sign(f_i(v))w_i H,
    \end{equation}
    where the matrix $H=p\trans{v}/\norm{v}$ is chosen such that $\norm{Zp}=\norm{Z}_2$ and $\norm{p}=1$.
    Then the norm-inequalities in \eqref{eq:eigvec_cond_def} are satisfied as equalities for both the spectral and Frobenius norm, and
    \begin{equation}
        \norm{v'} = \norm{Z}_2\left(\sum_{i=1}^m|f_i(v)|w_i\right)\norm{v}.
    \end{equation}
    Dividing by $\norm{v}$ shows the desired equality, and the result follows.
    % \qed
\end{proof}

Similar to the eigenvalue condition number, the perturbations in \Cref{def:eigvec_cond_definition} do not respect the symmetric structure of \eqref{eq:NEPv_scale_inv}.
Therefore, we now define the symmetric eigenvector condition number similar to \Cref{def:sym_eigval_cond_def}.

\begin{definition}\label{def:symm_eigvec_cond_definition}
    Assume $(\lambda, v)$ is a simple eigenpair of \eqref{eq:NEPv_scale_inv}, and let $(\lambda(\varepsilon),$$ v(\varepsilon))$ be the unique functions satisfying \eqref{eq:perturbed_problem_constraint}.
    Given the weights $w_i \geq 0$, $i = 1, \dots, m$, the symmetric normwise eigenvector condition number is defined as
    \begin{equation}\label{eq:symm_eigvec_cond_def}
        \kappa_{\text{sym}}^{(*)}(v) = \sup \left\{ \frac{\norm{v'}}{\norm{v}}\::\: \norm{E_i}_* \leq w_i, \quad E_i = \trans{E_i}, \quad i = 1, \dots, m \right\},
    \end{equation}
    where the supremum is taken over all feasible $E_i$, and where $*$ is either $2$ or $F$ for the two norm and Frobenius norm, respectively.

    %%%%%%%%%%% Old Definition %%%%%%%%%%%%%%%
    % Let $(\lambda, v)$ be a simple eigenpair of \eqref{eq:NEPv_scale_inv}.
    % Then the symmetric normwise condition number of $v$, $\kappa_{\text{sym}}(v)$, is defined to be
    % \begin{equation}\label{eq:symm_eigvec_cond_def}
    % \begin{split}
    %     \kappa_{\text{sym}}(v) = \limsup_{\varepsilon\rightarrow 0 }\left\{ \frac{\norm{v'(\varepsilon)}}{\norm{v}}\::\: \left(A(v(\varepsilon)) + \Delta A(v(\varepsilon))\right)v(\varepsilon) = \lambda(\varepsilon)v(\varepsilon), \right. \\
    %     \left. \phantom{\frac{|\Delta \lambda|}{\varepsilon |\lambda|}} \frac{1}{2}\trans{v(\varepsilon)}v(\varepsilon) = 1, \quad \norm{\Delta A_i}\leq \varepsilon w_i, \quad \trans{(\Delta A_i)} = \Delta A_i  \right\},
    % \end{split}
    % \end{equation}
    % where $w_i>0$, $i = 1,\dots,m$, are prescribed weights.
    %%%%%%%%%%%%%%%%%%%%%%%%%%%%%%%%%%%%%%%%%%%%%%
\end{definition}

As was the case for the eigenvalue condition number, it turns out that enforcing symmetric perturbations as is done in \Cref{def:symm_eigvec_cond_definition}, has no impact on the expression for the eigenvector condition number, when one uses the spectral norm to measure their sizes. Using the Frobenius norm however results in a smaller condition number.

\begin{theorem}
    Let $(\lambda,v)$ be as in \Cref{def:symm_eigvec_cond_definition}, then 
    \begin{equation}
        \kappa_{\text{sym}}^{(2)}(v) = \kappa(v), \quad \text{and} \quad \kappa_{\text{sym}}^{(F)}(v) = \frac{1}{\sqrt{2}} \kappa(v).
    \end{equation}
\end{theorem}

\begin{proof}
    For the matrix two norm, we still have the upper bound on the condition number derived from the formula for $v'$ in \Cref{lem:delta_v_lemma}.
    It remains to be shown that this bound can be attained by considering specific symmetric perturbations.
    Let $Z = V(\trans{V}(J(v)-\lambda I)V)^{-1}\trans{V}$, and take $p\in\mathbb{R}^n$ such that $\norm{Zp}=\norm{Z}_2$ and $\norm{p}=1$.
    Furthermore, let $\hat{v} = v/\norm{v}$, and define $H$ as the Householder reflector
    \begin{equation}
        H = I-2\frac{(p-\hat{v})\trans{(p-\hat{v})}}{\trans{(p-\hat{v})}(p-\hat{v})}.
    \end{equation}
    Let the perturbations $E_i$, $i=1,\dots,m$ be defined by
    \begin{equation}
        E_i = \sign(f_i(v))w_iH,
    \end{equation}
    then the perturbations are clearly symmetric, and the norm inequalities in \Cref{def:symm_eigvec_cond_definition} are satisfied as equalities.
    Substituting these perturbations into the expression for $v'$ \eqref{eq:deltav_expression} gives
    \begin{equation}
        v' = -Zp\sum_{i=1}^m|f_i(v)|w_i\norm{v}
    \end{equation}
    Taking the norm on both sides of this expression, and dividing by $\norm{v}$ gives the first stated result.

    For the Frobenius norm, a stricter upper bound can be obtained starting from the eigenvector sensitivity \eqref{eq:deltav_expression} as
    \begin{equation}\label{eq:eigvec_sensitivity_frobenius}
        \norm{v'} = \norm{Z \sum_{i=1}^m f_i(v) E_i v} \leq \sum_{i=1}^m  |f_i(v)| \norm{Z E_i v}.
    \end{equation}
    Let $p_i = \trans{Z}ZE_iv$, then
    \begin{align}
        \norm{Z E_i v}^2 = \trans{v} E_i p_i = \trace \left(E_i \frac{p_i \trans{v} + v \trans{p_i}}{2} \right) & \leq w_i \norm{v} \norm{p_i} \beta \\
        & \leq w_i \norm{v} \norm{Z}_2 \norm{Z E_i v} \beta,
    \end{align}
    similar to the derivations in \eqref{eq:trace_trick} and \eqref{eq:rank_2_symm_matrix_frob_norm}. Here, $\beta$ is defined as in \eqref{eq:beta_factor} with $\theta$ the angle between $p_i$ and $v$, but as these vectors are orthogonal, $\beta$ must be equal to $1/\sqrt{2}$, hence
    \begin{equation}
        \norm{Z E_i v} \leq \frac{1}{\sqrt{2}}w_i \norm{v} \norm{Z}_2.
    \end{equation}
    Using this inequality in \eqref{eq:eigvec_sensitivity_frobenius} and dividing by $\norm{v}$ yields
    \begin{equation}
        \frac{\norm{v'}}{\norm{v}} \leq \frac{1}{\sqrt{2}} \norm{Z}_2 \sum_{i=1}^m |f_i(v)| w_i = \frac{1}{\sqrt{2}} \kappa(v),
    \end{equation}
    which gives an upper bound on $\kappa_{sym}^{(F)}(v)$. We show that this bound is attained for the symmetric perturbations
    \begin{equation}
        E_i = \sign(f_i(v)) w_i H, \quad \text{with} \quad H = \frac{p\trans{v} + v \trans{p}}{\sqrt{2}\norm{v}},
    \end{equation}
    and where $p$ is a vector such that $\norm{p} = 1$ and $\norm{Zp} = \norm{Z}_2$. Note that $p$ is the right singular vector of $Z$ corresponding to its largest singular value, and since $v$ is orthogonal to $Z$, it is also orthogonal to $p$. As a result,
    \begin{equation}
        \norm{E_i}_F = w_i \norm{H}_F = w_i,
    \end{equation}
    and 
    \begin{align}
        \norm{v'} & = \left\| Z \sum_{i=1}^m |f_i(v)| w_i H v \right\| = \norm{ZHv} \sum_{i=1}^m |f_i(v)| w_i, \\
        & = \frac{\norm{v}}{\sqrt{2}} \norm{Zp} \sum_{i=1}^m |f_i(v)| w_i = \frac{\norm{v}}{\sqrt{2}} \norm{Z}_2 \sum_{i=1}^m |f_i(v)| w_i.
    \end{align}
    Dividing by $\norm{v}$ gives the upper bound, which concludes the proof.
    % \qed
\end{proof}

\begin{remark}
    For linear eigenvalue problems, the eigenvector condition number is closely related to the \textit{spectral gap} of the corresponding eigenvalue \cite{higham1998structured}\cite{Saad:2011:EIGBOOK}.
    That is, the eigenvector conditioning (for linear symmetric problems) is inversely proportional to the distance between the eigenvalue associated with $v$ and the next closest eigenvalue.
    It turns out that a similar statement is true for some special cases of \eqref{eq:NEPv_scale_inv}.
    If $J(v)$ is diagonalizable with eigenvalues $\lambda, \lambda_1, \dots, \lambda_{n-1}$ and corresponding left eigenvectors $u, u_1, \dots, u_{n-1}$, then the eigenvector condition number \eqref{eq:eigvec_cond_closed_form} can be expressed as
    \begin{align*}
        \kappa(v) = \left\| \sum_{j=1}^{n-1} \frac{\Tilde{u}_j \trans{u_j}}{\lambda_j - \lambda} \right\|_2 \left(\sum_{i=1}^m|f_i(v)|w_i \right),
    \end{align*}
    with $\Tilde{u}_j$ the $j$th column of $U(\trans{U}U)^{-1}$ where $U = [u_1, \dots, u_{n-1}]$.
    To see this, note that both $V$ and $U$ are full rank matrices in $\mathbb{C}^{n \times (n-1)}$ with the eigenvector $v$ in their left nullspaces.
    This implies that there exists an invertible matrix $T \in \mathbb{C}^{(n-1) \times (n-1)}$ such that $V = UT$. Substituting this in \eqref{eq:eigvec_cond_closed_form} and using the fact that $\trans{U}J(v) = \Lambda \trans{U}$ with $\Lambda = \diag(\lambda_1,\dots,\lambda_{n-1})$ gives
    \begin{align*}
        \kappa(v) = \left\| U(\trans{U}U)^{-1} (\Lambda - \lambda I)^{-1} \trans{U} \right\|_2 \left(\sum_{i=1}^m|f_i(v)|w_i \right).
    \end{align*}
    Moreover, in case %the two-norm is used and 
    $J(v)$ is symmetric, we get
    \begin{align*}
        \kappa(v) = \frac{1}{|\lambda_{\hat{\jmath}} - \lambda|} \left(\sum_{i=1}^m|f_i(v)|w_i \right),
    \end{align*}
    where $\lambda_{\hat{\jmath}}$ is the eigenvalue closest to $\lambda$, similar to the eigenvector conditioning of standard symmetric eigenvalue problems \cite[equation (3.45)]{Saad:2011:EIGBOOK}.
    However, the property that $J(v)$ is symmetric is not guaranteed, even if $A(v)$ is symmetric, meaning the connection between the inverse spectral gap and the eigenvector condition number is generally more complicated than for linear problems.
\end{remark}

\section{Backward error}\label{sec:back_err}

Having derived expressions for different quantities related to the conditioning of \eqref{eq:NEPv_scale_inv} in the previous sections, we now proceed to develop the other important concept considered in this work, namely backward errors for \eqref{eq:NEPv_scale_inv}.
Again, as for the eigenvalue and eigenvector condition numbers, we consider perturbations to \eqref{eq:NEPv_scale_inv} of the form
\begin{equation}
    \Delta A(v, \varepsilon) = \varepsilon \sum_{i=1}^m f_i(v) E_i,
\end{equation}
where the matrices $E_i$ will satisfy norm constraints, and the parameter $\varepsilon$ corresponds to the size of the perturbation.
We begin by defining the backward error associated with an approximate eigenpair $(\tilde{\lambda}, \tilde{v})$, and subsequently derive a computable expression for this quantity.

\begin{definition}\label{def:back_err_nonlin_def}
Let $(\tilde{\lambda}, \tilde{v})$ be an approximate eigenpair of \eqref{eq:NEPv_scale_inv},
then the backward error associated with this pair is defined as
    \begin{equation}\label{eq:back_error_def}
        \eta(\tilde{\lambda}, \tilde{v}) = \min \left\{ |\varepsilon|  \::\: \left(A(\tilde{v}) + \Delta A(\tilde{v}, \varepsilon)\right)\tilde{v} = \tilde{\lambda}\tilde{v},\: \norm{E_i} \leq w_i \right\},
    \end{equation}
    with $w_i\geq0$, $i=1,\dots,m$ prescribed weights, and where the minimum is taken over $\varepsilon$ and the matrices $E_i$.
\end{definition}

% We derive a computable expression for the backward error.
\begin{theorem}\label{thm:back_err_nonlin}
    Let $(\tilde{\lambda}, \tilde{v})$ be an approximate eigenpair of \eqref{eq:NEPv_scale_inv}, and let $r=A(\tilde{v})\tilde{v}-\tilde{\lambda}\tilde{v}$ be the associated residual.
    Then the backward error \eqref{eq:back_error_def} is given by
    \begin{equation}\label{eq:backward_error_formula}
         \eta(\tilde{\lambda}, \tilde{v})  = \frac{\norm{r}}{\sum_{i=1}^m|f_i(\tilde{v})|w_i\norm{\tilde{v}}}.
    \end{equation}
\end{theorem}

\begin{proof}
    It is straight-forward to establish the lower bound given by
    \begin{equation}\label{eq:back_err_lower_bound}
        |\varepsilon| \geq \frac{\norm{r}}{\sum_{i=1}^m|f_i(\tilde{v})|w_i\norm{\tilde{v}}}.
    \end{equation}
    This bound is attained from the first constraint when considering the feasible perturbations
    \begin{equation}\label{eq:back_err_perts}
        E_i = -\sign(f_i(\tilde{v})) w_i \frac{r\trans{\tilde{v}}}{\norm{r}\norm{\tilde{v}}} \qquad \text{and} \qquad \varepsilon = \frac{\norm{r}}{\sum_{i=1}^m|f_i(\tilde{v})|w_i\norm{\tilde{v}}},
    \end{equation}
    from which the statement follows.
    % \qed
\end{proof}

Similar to the first definition of conditioning of the eigenvalues of \eqref{eq:NEPv_scale_inv}, \Cref{def:back_err_nonlin_def} does not respect the symmetric structure of \eqref{eq:NEPv_scale_inv}. 
To remedy this, we define a backward error for which we require the optimal perturbations to be symmetric. 

\begin{definition}\label{def:symm_back_err_nonlin_def}
Let $(\tilde{\lambda}, \tilde{v})$ be an approximate eigenpair of \eqref{eq:NEPv_scale_inv}.
Then the symmetric backward error associated with this pair is defined as
    \begin{equation}\label{eq:back_error_def_sym}
        \eta_{\text{sym}}^{(*)}(\tilde{\lambda}, \tilde{v}) = \min \left\{ |\varepsilon| \::\: \left(A(\tilde{v}) + \Delta A(\tilde{v}, \varepsilon)\right)\tilde{v} = \tilde{\lambda}\tilde{v},\: \norm{E_i}_* \leq  w_i,\: \trans{E_i} = E_i \right\},
    \end{equation}
    with $w_i\geq 0$, $i=1,\dots,m$ prescribed weights, where the minimum is taken over $\varepsilon$ and the matrices $E_i$, and where $*$ is either 2 or $F$ for the spectral and Frobenius norm, respectively.
\end{definition}

Analogous to the relation between the condition number $\kappa(\lambda)$ and the symmetric condition number $\kappa_{\text{sym}}(\lambda)$, it turns out that the backward error remains unchanged when requiring the perturbations to be symmetric and when their sizes are measured by the spectral norm. Using the Frobenius norm on the other hand, results in the inequality $\eta_{\text{sym}}^{(F)}(\Tilde{\lambda}, \Tilde{v}) \geq \eta(\Tilde{\lambda}, \Tilde{v})$. We state these relations more precisely in \Cref{thm:symm_back_err_nonlin}.

\begin{theorem}\label{thm:symm_back_err_nonlin}
    Let $(\tilde{\lambda},\tilde{v})$ and $r$ be as in \Cref{thm:back_err_nonlin} and let $\vartheta$ be the angle between $r$ and $\tilde{v}$, i.e., $\vartheta \coloneq \arccos((\trans{r}\tilde{v})/(\norm{r}\norm{\tilde{v}}))$, then
    \begin{equation}
        \eta_{\text{sym}}^{(2)}(\tilde{\lambda}, \tilde{v}) = \eta(\tilde{\lambda}, \tilde{v}), \quad \text{and} \quad \eta_{\text{sym}}^{(F)}(\tilde{\lambda}, \tilde{v}) = \eta(\tilde{\lambda}, \tilde{v}) \gamma,
    \end{equation}
    where
    \begin{equation}
        \gamma \coloneq \sqrt{1 + \sin^2(\vartheta)}.
    \end{equation}
    %Then $\eta_{\text{sym}}(\tilde{\lambda}, \tilde{v}) = \eta(\tilde{\lambda}, \tilde{v})$, if $\norm{\cdot}_2$ is used in Definition~\ref{def:symm_back_err_nonlin_def}.
\end{theorem}
\begin{proof}
    For the backward error where the perturbations are bounded using the spectral norm, we still have the lower bound \eqref{eq:back_err_lower_bound}. 
    It remains to be shown that this bound can be attained by a symmetric perturbation.
    % Let $r = A(\tilde{v})\tilde{v}-\tilde{\lambda}\tilde{v}$ be the residual associated with the pair $(\tilde{\lambda},\tilde{v})$.
    As in the proof of \Cref{thm:sym_eigval_cond_thm}, let $H\in\mathbb{R}^{n\times n}$ be the Householder reflector defined by the vector $(r/\norm{r}-\tilde{v}/\norm{\tilde{v}})$.
    Then the matrix $H$ has the property
    \begin{equation}
        H\tilde{v} = \frac{\norm{\tilde{v}}}{\norm{r}}r.
    \end{equation}
    Consider then the feasible perturbations defined by 
    \begin{equation}
        E_i = -\sign(f_i(\tilde{v}))w_iH.
    \end{equation}
    Clearly, $E_i$ is symmetric for $i=1,\dots,m$, and it is easy to verify that these perturbations satisfy the constraints in \Cref{def:symm_back_err_nonlin_def} with $\varepsilon$ equal to the bound in \eqref{eq:back_err_lower_bound}, which gives the stated result.

    In case the Frobenius norm is used to bound the perturbations, we start from the equality in \Cref{def:symm_back_err_nonlin_def}
    \begin{equation}
        r = - \varepsilon \sum_{i=1}^m f_i(\tilde{v}) E_i \tilde{v},
    \end{equation}
    and use it to write 
    \begin{align}
        \norm{r}^2 \gamma^2 
        & = \norm{r}^2 (2 - \cos^2(\theta)) = \left( 2\trans{r} - \frac{\trans{r} \tilde{v}}{\norm{\tilde{v}}^2} \trans{\tilde{v}} \right) r
        \\
        & = - \varepsilon \left( 2\trans{r} - \frac{\trans{r} \tilde{v}}{\norm{\tilde{v}}^2} \trans{\tilde{v}} \right) \sum_{i=1}^m f_i(\tilde{v})E_i \tilde{v}
        \\
        & = \varepsilon \sum_{i=1}^m f_i(\tilde{v}) \trace\left( E_i \left( \frac{\trans{r} \tilde{v}}{\norm{\tilde{v}}^2}\tilde{v} \trans{\tilde{v}} - \tilde{v}\trans{r} - r \trans{\tilde{v}} \right) \right)
        \\
        & \leq |\varepsilon| \sum_{i=1}^m |f_i(\tilde{v})| w_i \norm{\tilde{v}} \norm{r} \gamma,
    \end{align}
    where we used the fact that
        $\left\| \frac{\trans{r} \tilde{v}}{\norm{\tilde{v}}^2}\tilde{v} \trans{\tilde{v}} - \tilde{v}\trans{r} - r \trans{\tilde{v}} \right\|_F = \norm{\tilde{v}} \norm{r} \gamma.$
    By rewriting the inequality, we get the lower bound
    \begin{equation}\label{eq:back_err_frob_sym_lower_bound}
        |\varepsilon| \geq \frac{\norm{r} \gamma}{\sum_{i=1}^m |f_i(\tilde{v})| w_i \norm{\tilde{v}}} = \eta(\tilde{\lambda}, \tilde{v}) \gamma.
    \end{equation}
    The constraints in \Cref{def:symm_back_err_nonlin_def} are satisfied for the symmetric perturbations
    \begin{equation}
        E_i = \frac{\sign(f_i(\tilde{v})) w_i}{\gamma \norm{r}\norm{\tilde{v}}}\left( \frac{\trans{r} \tilde{v}}{\norm{\tilde{v}}^2}\tilde{v} \trans{\tilde{v}} - \tilde{v}\trans{r} - r \trans{\tilde{v}} \right),  
    \end{equation}
    and $\varepsilon$ equal to the lower bound in \eqref{eq:back_err_frob_sym_lower_bound}, which concludes the proof.
    % \qed
\end{proof}

Several iterative methods for \eqref{eq:NEPv_scale_inv} only compute approximations to the eigenvector.
In such cases, the backward error may be more suitably defined as
\begin{equation}
    \eta(\tilde{v}) = \min_{\tilde{\lambda}} \eta(\tilde{\lambda}, \tilde{v}).
\end{equation}
It is common practice to compute an approximation to an eigenvalue from an approximation of an eigenvector by means of the Rayleigh-quotient.
The following theorem shows that this strategy is optimal.

\begin{theorem}\label{thm:eigvec_back_err}
    Suppose $\tilde{v}$ is an approximate eigenvector of \eqref{eq:NEPv_scale_inv}.
    Then the eigenvector backward error is
    \begin{equation}
        \eta(\tilde{v}) = \min_{\tilde{\lambda}} \eta(\tilde{\lambda}, \tilde{v}) = \frac{\norm{r_\ast}}{\sum_{i=1}^m|f_i(\tilde{v})|w_i\norm{\tilde{v}}},
    \end{equation}
    where $r_\ast$ is the residual associated with the approximate eigenvalue $\tilde{\lambda}$ given by the Rayleigh quotient, i.e., 
    \begin{equation}
        \tilde{\lambda} := \frac{\trans{\tilde{v}} A(\tilde{v}) \tilde{v}}{\trans{\tilde{v}}\tilde{v}}, \quad r_\ast := A(\tilde{v})\tilde{v} - \tilde{\lambda}\tilde{v}.
    \end{equation}
\end{theorem}
\begin{proof}
    We have that 
    \begin{equation}
        \eta(\tilde{v}) = \min_{\tilde{\lambda}} \eta(\tilde{\lambda}, \tilde{v}) = \frac{1}{\sum_{i=1}^m|f_i(\tilde{v})|w_i\norm{\tilde{v}}} \min_{\tilde{\lambda}} \norm{\tilde{\lambda} \tilde{v} - A(\tilde{v})\tilde{v}}.
    \end{equation}
    Since $\tilde{v}$ is fixed and $A(\tilde{v})$ is symmetric, the minimum on the right is achieved by the Rayleigh quotient \cite[p. 453]{Golub:2013:MATRIX}.
    % \qed
\end{proof}

\begin{remark}[Eigenvalue and eigenvector backward errors]
    In contrast to generalized linear \cite{higham1998structured} and eigenvalue-nonlinear problems \cite{tisseur2000backward}, we can derive computable expressions for $\eta(\tilde{v})$, i.e., \Cref{thm:eigvec_back_err}.
    For linear and eigenvalue-nonlinear problems, there is no closed form for the analogous quantity.
    However, it seems that the eigenvalue backward error (minimizing $\eta(\tilde{\lambda}, \tilde{v})$ over all $\tilde{v}$) is in turn an intractable quantity in our setting, while it is possible to derive closed form expressions for this quantity in the context of eigenvalue (non)linear problems. 
\end{remark}

\begin{remark}[Eigenvalue extraction]
    \Cref{thm:eigvec_back_err} has important implications for iterative methods for \eqref{eq:NEPv_scale_inv} that compute eigenvector approximations.
    For other classes of linear and nonlinear eigenvalue problems, different strategies for extracting eigenvalue approximations from eigenvector approximations have been considered, see, e.g., \cite{paige1995harmritz}\cite{hochtenbach2008harmritz}, which aim to improve the approximations of eigenvalues in different situations.
    \Cref{thm:eigvec_back_err} shows that if the aim is to minimize the backward error associated with an eigenvector approximation $\tilde{v}$, then the Rayleigh-quotient extraction is optimal for \eqref{eq:NEPv_scale_inv}.
\end{remark}
%%%%%%%%%%%%%%%%%%%%%%%%%%%%%%%%%%%%%%%%%%%%%%%%%
%%%%%%%%%%%%%%%%%%%%%%%%%%%%%%%%%%%%%%%%%%%%%%%%%
%%%%%%%%%%%%%%%%%%%%%%%%%%%%%%%%%%%%%%%%%%%%%%%%%

\section{Numerical illustrations}\label{sec:sims}

\subsection{Ill-conditioning by increasing problem size}\label{sim:wilkinson}

In order to illustrate our findings regarding the condition number \eqref{eq:condition_number}, and how its properties differ from the corresponding theory for linear problems, we consider a specific instance of \eqref{eq:NEPv_scale_inv}.
Let $A_0, A_k, B_k \in\mathbb{R}^{n \times n}$, $k=1,\dots,n-1$, be symmetric matrices.
Consider the class of symmetric quadratic problems defined by
\begin{equation}\label{eq:wilkinson_nepv}
    A(v)v := \left(A_0 + \sum_{k=1}^{n-1}f_k(v)A_k\right)v = \lambda v,
\end{equation}
where the scaling-invariant functions $f_k$ are given as
\begin{equation}
    f_k(v) = \frac{\trans{v}B_kv}{\trans{v}v}.
\end{equation}
Similar structures were recently studied in \cite{janssens2025linearizing}, where they also arise from the discretization of a certain nonlinear wave-equation.
Let $A_0$ be given as $A_0 = \diag(n, n-1, \dots, 2, 1) + M$, where $M$ is defined as
\begin{equation}
    M = \begin{bmatrix}
        -1 & -\trans{\mathbbm{1}} & 0\\
        -\mathbbm{1} & 0 & 0 \\
        0 & 0 & 0
    \end{bmatrix} \in\mathbb{R}^{n \times n}.
\end{equation}
Here, $\mathbbm{1}\in\mathbb{R}^{n-2}$ denotes a vector of all ones.
Let $A_k$ be all zeros except for $(A_k)_{1,k} = (A_k)_{k,1}$ which we set equal to one.
Let $B_k$ be such that $(B_k)_{1,1}=1$.
Set $(B_k)_{1,k+1} = (B_k)_{k+1,1} = n/2$, and set the remaining elements of the first row and column of $B_k$ to zero. 
Finally, the remaining elements of $B_k$ can be chosen arbitrarily.
For instance, $A(v)$ in \eqref{eq:wilkinson_nepv} for $n=3$ becomes
\begin{equation}
    A(v) = \begin{bmatrix}
        2 & -1 & 0 \\
        -1 & 2 & 0 \\
        0 & 0 & 1
    \end{bmatrix} + f_1(v)\begin{bmatrix}
        1\:\: & 0\:\: & 0 \\
        0\:\: & 0\:\: & 0 \\
        0\:\: & 0\:\: & 0 \\
    \end{bmatrix} + f_2(v)\begin{bmatrix}
        0\:\: & 1\:\: & 0 \\
        1\:\: & 0\:\: & 0 \\
        0\:\: & 0\:\: & 0 \\
    \end{bmatrix},
\end{equation}
with $B_1$ and $B_2$ given by
\begin{equation}
    B_1 = \begin{bmatrix}
        1 & 3/2 & 0 \\
        3/2 & \cdot & \cdot \\
        0 & \cdot & \cdot
    \end{bmatrix},\: B_2 = \begin{bmatrix}
        1 & 0 & 3/2 \\
        0 & \:\cdot & \cdot \\
        3/2 & \:\cdot & \cdot
    \end{bmatrix},
\end{equation}
where the dots $(\:\cdot\:)$ denote arbitrary values that retain the symmetry of $B_k$.

\begin{table}[]
\centering
\caption{Relative eigenvalue condition number, the absolute value of the inner product of the left and right eigenvectors, and $\alpha := \sum_{i=1}^m|f_i(v)|\norm{A_i}\norm{v}\norm{u}$, corresponding to the eigenvalue $\lambda=n$ of \eqref{eq:wilkinson_nepv} for increasing $n$. For an approximate solution $(\tilde{\lambda}, \tilde{v})$ of \eqref{eq:wilkinson_nepv}, the three rightmost columns give the backward error, forward error, and an estimation of the forward error.}
\label{tab:wilkinson_cond_num}
\begin{tabular}{ccccccc}
\toprule
$n$ & $\kappa(\lambda)$ & $|\trans{u}v|$     & $\alpha$ & $\eta(\tilde{\lambda}, \tilde{v})$ & $|\tilde{\lambda} - n|/n$ & $\eta(\tilde{\lambda}, \tilde{v}) \cdot \kappa(\lambda)$ \\ 
\midrule
\ph{0}2 & $2.2 \cdot 10^{0\ph{0}}$ & $4.4\cdot10^{-1\ph{0}}$ & \ph{0}2.0 & $6.3 \cdot 10^{-11}$ & $1.3 \cdot 10^{-10}$      & $1.4 \cdot 10^{-10}$ \\ 
\ph{0}5 & $6.7\cdot10^1\ph{0}$     & $2.7\cdot10^{-2\ph{0}}$ & \ph{0}9.4 & $1.8 \cdot 10^{-11}$ & $6.8 \cdot 10^{-10}$      & $1.2 \cdot 10^{-9\ph{0}}$ \\ 
10      & $9.0\cdot10^3\ph{0}$     & $2.2\cdot10^{-4\ph{0}}$ & 19.9      & $1.3 \cdot 10^{-11}$ & $6.3 \cdot 10^{-8\ph{0}}$ & $1.2 \cdot 10^{-7\ph{0}}$ \\ 
20      & $1.7\cdot10^8\ph{0}$     & $1.1\cdot10^{-8\ph{0}}$ & 40.4      & $7.5 \cdot 10^{-12}$ & $1.0 \cdot 10^{-4\ph{0}}$ & $1.3 \cdot 10^{-3\ph{0}}$ \\ 
30      & $3.4\cdot10^{12}$        & $5.9\cdot 10^{-13}$     & 60.6      & $5.8 \cdot 10^{-12}$ & $3.3 \cdot 10^{-2\ph{0}}$ & $2.0 \cdot 10^{1\ph{-0}}$ \\ \bottomrule
\end{tabular}
\end{table}

With these choices of matrices, it is straight-forward to verify that for every $n$, \eqref{eq:wilkinson_nepv} has the eigenpair $(\lambda, v) = (n,e_1)$, where $e_1\in\mathbb{R}^n$ is the first standard unit vector.
In \Cref{tab:wilkinson_cond_num} we record the relative eigenvalue condition number (i.e., we set $w_i=\norm{A_i}$ in \Cref{thm:sym_eigval_cond_thm}) of this eigenvalue for a few values of $n$.
Notice that the condition number does not include perturbations on $B_k$.
Clearly, we observe a dramatic growth in the condition number of this eigenvalue as $n$ increases.
In order to explain this behavior, we reference \Cref{thm:eigval_cond_thm}, and compute the Jacobian $J(v)$ for this problem.
One can readily determine that the Jacobian is given by 
\begin{equation}
    J(v) = A(v) + \sum_{k=1}^{n-1}A_kv\frac{\partial}{\partial v}f_k(v),\quad \text{with} \quad \frac{\partial}{\partial v} f_k(v) = \frac{2\trans{v}B_k}{\trans{v}v} - \frac{2\trans{v}B_kv}{(\trans{v}v)^2}\trans{v}.
\end{equation}
Now, evaluating $J(v) = J(e_1)$ reveals that 
\begin{equation}\label{eq:wilk_matrix}
    J(e_1) = \begin{bmatrix}
        n & n &  &  &  &  \\
          & n-1 & n &  &  & \\
          & & n-2 & n & & \\
          & & & \ddots & \ddots & \\
          & & & & 2 & n \\
          & & & & & 1
    \end{bmatrix} =: W_n,
\end{equation}
where $W_n$ is the $n$th so-called Wilkinson matrix \cite[Chapter 2]{Wilkinson:1988:ALGEBRAIC}. % (Page 90)
Wilkinson used this family of matrices as an example of ``extremely ill-conditioned matrices'', due to the left and right eigenvectors of \eqref{eq:wilk_matrix} being very close to orthogonal.
In fact, the inner product of the left and right eigenvectors associated with $\lambda = n$ is explicitly given by the expression $|\trans{u}v| = (n-1)!/n^{n-1}$ \cite{Wilkinson:1988:ALGEBRAIC}, which can be estimated as 
\begin{equation}
    |\trans{u}v|\approx \frac{\sqrt{2\pi n}}{e^n},
\end{equation}
for large $n$ using Stirling's formula.
Hence, by \Cref{thm:eigval_cond_thm}, the eigenvalue $\lambda = n$ of \eqref{eq:wilkinson_nepv} will become exponentially ill-conditioned as $n$ increases, which explains the dramatic behavior observed in \Cref{tab:wilkinson_cond_num}.
This example demonstrates that even a simple symmetric problem such as \eqref{eq:wilkinson_nepv} can have highly ill-conditioned eigenvalues, in contrast to the linear case, where symmetric problems always have perfectly conditioned eigenvalues, and that \eqref{eq:NEPv_scale_inv} warrants additional care in this respect.

For an approximated solution $(\tilde{\lambda}, \tilde{v})$, we also added the backward error $\eta(\tilde{\lambda}, \tilde{v})$ \eqref{eq:backward_error_formula}, the forward error $|\tilde{\lambda} - n|/n$, and an estimation $\eta(\tilde{\lambda}, \tilde{v}) \cdot \kappa(\lambda)$ of the forward error in the three rightmost columns of \Cref{tab:wilkinson_cond_num}. The approximation $(\tilde{\lambda}, \tilde{v})$ is computed by perturbing the coefficient matrix $A_0$ in \eqref{eq:wilkinson_nepv} with a small random symmetric matrix and solving it using inverse iteration ($J$-version) \cite{jarlebring2014inverseiter}. While the backward error remains small, the forward error on the eigenvalue gradually increases with $n$ which is expected from the ill-conditioning of the problem. The expression $\eta(\tilde{\lambda}, \tilde{v}) \cdot \kappa(\lambda)$ gives an upper bound on the forward error.

Finally, we remark that while the specific values of the non-zero elements in $A_0, A_k$ and $B_k$ are of course chosen in order to illustrate possible complications that can arise with respect to conditioning, we do not believe the structure of these matrices to be unrealistic.
For instance, bordered structures such as the ones found in $A(v)$ for this example could be achieved by considering standard finite difference discretizations of a PDE over a square domain, with a mixture of homogeneous and inhomogeneous Dirichlet boundary conditions.
Hence, while we do not expect this exact example to be encountered in application, we do believe it to be a representative structure, and that it illustrates that even seemingly simple problems have to be approached with caution when incorporating eigenvector nonlinearities.

\subsection{Eigenvalue bifurcation}\label{sim:eigval_illcond}

%\begin{figure}
%   \centering
%   \input{bifurcationplot}
%   \caption{Eigenvalues of \eqref{eq:bifurcation_problem} as functions of a perturbation parameter $\delta$. The different marks correspond to different branches of the eigenvalues.}
%   \label{fig:bifurcation}
%\end{figure}

%\begin{figure}
%   \centering
%   \input{bifurcationplot_condition}
%   \caption{Condition number of each eigenvalue of \eqref{eq:bifurcation_problem} as functions of a perturbation parameter $\delta$. The different marks correspond to the different branches in \Cref{fig:bifurcation}.}
%   \label{fig:bifurcation_condition}
%\end{figure}

The example in \Cref{sim:wilkinson} demonstrated that even though \eqref{eq:NEPv_scale_inv} is symmetric, one can encounter eigenvalues that are arbitrarily ill-conditioned by increasing the size of the problem.
However, for any given $n$ in \eqref{eq:wilkinson_nepv}, the eigenvalue $\lambda=n$ has a finite condition number.
To further reiterate that one can encounter ill-conditioning even for seemingly simple instances of \eqref{eq:NEPv_scale_inv}, we now illustrate (for a fixed problem size) that even though we have a symmetric problem, a saddle-node bifurcation may occur in its eigenvalues, when they are viewed as functions of a perturbation parameter.
In this sense, this implies that the eigenvalue condition number can become infinitely large.
%Consider the NEPv \eqref{eq:NEPv_scale_inv} given by

\begin{figure}[h]
    \centering
    \begin{subfigure}
        \centering
        \tikzsetnextfilename{bifurcationplot} % name for the temporary generated file by tikz externalization in gfx_tmp (see preamble)
\begin{tikzpicture} % vipl: changing this slightly because im too pedantic about plots :)
    \begin{axis}[
            grid=major,
            %title = Bifurcation diagram, % vipl was here, no title for plots, just caption
            xlabel={$\delta$},
            ylabel={$\lambda$},
            %width=0.95\textwidth,
            %width=\textwidth,
            width=0.49\linewidth,
            height=0.5\textwidth,
            xmin=-5,
            xmax=0,
            ymin=-6,
            ymax=6,
        ]
        \addplot [
            black,
            thick,
            mark=square,
            mark repeat=50,
            mark options={color=mydarkred,scale=1.2,solid},
            ] table [x=a, y=b, col sep=comma]  {data/test_condition_bifurcation4.csv};
        
        \addplot [
            black,
            thick,
            mark=o,
            mark repeat=50,
            mark options={color=myred,scale=1.2,solid},
            ] table [x=a, y=b, col sep=comma]  {data/test_condition_bifurcation1.csv};
        
        \addplot [
            black,
            thick,
            mark=x,
            mark repeat=50,
            mark phase=35,
            mark options={color=mygreen,scale=1.5,solid},
            ] table [x=a, y=b, col sep=comma]  {data/test_condition_bifurcation2.csv};
        
        \addplot [
            black,
            thick,
            mark=diamond,
            mark repeat=50,
            mark phase=40,
            mark options={color=myblue,scale=1.5,solid},
            ] table [x=a, y=b, col sep=comma]  {data/test_condition_bifurcation3.csv};
        
        \addplot [
            black,
            thick,
            mark=triangle,
            mark repeat=50,
            mark options={color=mydarkblue,scale=1.5,solid},
            ] table [x=a, y=b, col sep=comma]  {data/test_condition_bifurcation5.csv};
    \end{axis}
\end{tikzpicture}
        %\caption{Eigenvalues of \eqref{eq:bifurcation_problem} as functions of a perturbation parameter $\delta$. The different marks correspond to different branches of the eigenvalues.}
    \end{subfigure}
    \hfill
    \centering
    \begin{subfigure}
        \centering
        \tikzsetnextfilename{bifurcationplot_condition} % name for the temporary generated file by tikz externalization in gfx_tmp (see preamble) (vipl was here)
\begin{tikzpicture}
    \begin{semilogyaxis}[
            grid=major,
            xlabel={$\delta$},
            ylabel={$\kappa(\lambda)$},
            width=0.49\linewidth,
            height=0.5\textwidth,
            xmin=-5,
            xmax=0,
            ymin=1,
            ymax=1e3,
        ]
        \addplot [
            black,
            thick,
            mark=square,
            mark repeat=50,
            mark options={color=mydarkred,scale=1.2,solid},
            ] table [x=a, y=c, col sep=comma]  {data/test_condition_bifurcation4.csv};
        
        \addplot [
            black,
            thick,
            mark=o,
            mark repeat=50,
            mark options={color=myred,scale=1.2,solid},
            ] table [x=a, y=c, col sep=comma]  {data/test_condition_bifurcation1.csv};
        
        \addplot [
            black,
            thick,
            mark=x,
            mark repeat=40,
            mark phase=0,
            mark options={color=mygreen,scale=1.5,solid},
            ] table [x=a, y=c, col sep=comma]  {data/test_condition_bifurcation2.csv};
        
        \addplot [
            black,
            thick,
            mark=diamond,
            mark repeat=50,
            mark phase=40,
            mark options={color=myblue,scale=1.5,solid},
            ] table [x=a, y=c, col sep=comma]  {data/test_condition_bifurcation3.csv};
        
        \addplot [
            black,
            thick,
            mark=triangle,
            mark repeat=50,
            mark options={color=mydarkblue,scale=1.5,solid},
            ] table [x=a, y=c, col sep=comma]  {data/test_condition_bifurcation5.csv};
    \end{semilogyaxis}
    
\end{tikzpicture}
        %\caption{Condition number of each eigenvalue of \eqref{eq:bifurcation_problem} as functions of a perturbation parameter $\delta$. The different marks correspond to the different branches in \Cref{fig:bifurcation}.}
        \label{fig:bifurcation_condition}
    \end{subfigure}
    \caption{Eigenvalues of \eqref{eq:bifurcation_problem} (left) and corresponding relative eigenvalue condition numbers (right) as functions of a perturbation parameter $\delta$. The different marks correspond to the different branches of the eigenvalues.}
    \label{fig:joined_eigval_bif_plot}
\end{figure}

Consider another instance of \eqref{eq:NEPv_scale_inv}, that is of the same class as the problem in \Cref{sim:wilkinson}, but this time given by
\begin{equation}\label{eq:bifurcation_problem}
   A(v) = A_0 + \frac{\trans{v} B v}{\trans{v} v} A_1.
\end{equation}
We choose the matrices $A_0, A_1$ and $B$ as
\begin{align*}
   A_0 = \begin{bmatrix}
   1 + \delta & \phantom{-}1 & \phantom{-}1 \\ 
   1 & -2 & -2 \\
   1 & -2 & \phantom{-}0
   \end{bmatrix}, \quad A_1 = \begin{bmatrix}
   0 & \phantom{-}1 & \phantom{-}0 \\ 
   1 & \phantom{-}2 & -1 \\
   0 & -1 & \phantom{-}5
   \end{bmatrix}, \quad B = \begin{bmatrix}
   \phantom{-}0 & -1 & \phantom{-}2 \\ 
   -1 & \phantom{-}2 & \phantom{-}1 \\
   \phantom{-}2 & \phantom{-}1 & \phantom{-}1
   \end{bmatrix},
\end{align*}
with $\delta\in\mathbb{R}$ a perturbation parameter.
With these choices of matrices, \eqref{eq:bifurcation_problem} is solved using the method in \cite{janssens2025linearizing} for various values of $\delta$ ranging between $-5$ and $0$, and for each $\delta$ we record the real-valued eigenvalues of \eqref{eq:bifurcation_problem}.
%, and for each value we find 9 eigenvalues in $\mathbb{C}$.
%However, only a subset of these are real.
In \Cref{fig:joined_eigval_bif_plot}, these
%the real-valued eigenvalues of this problem
are displayed as functions of the parameter $\delta$.
For $\delta = -3.9$ and $\delta = -1.35$, two of the eigenvalues collide
%and form a complex conjugate pair 
around $\lambda = -2.94$ and $\lambda = 0.785$, respectively, indicating a saddle-node bifurcation.
In these two cases, the left eigenvector $u$ of the Jacobian is orthogonal to the eigenvector $v$.
In other words, the Jacobian \eqref{eq:eigvec_cond_mat_prelims} becomes singular, and the implicit function theorem no longer applies.
This means that the eigenpairs are no longer simple according to \Cref{def:simple_eig}, or equivalently, that the condition number is infinite.
Note that the slope of the curves in the left plot of \Cref{fig:joined_eigval_bif_plot} does not give the condition number of the corresponding eigenvalue,
but it does provide the eigenvalue sensitivity \eqref{eq:lambda_deriv_expression} for a specific choice of perturbations $E_i$.
Hence, the absolute value of the slope of the curves does provide a lower bound on the (absolute) condition number. 
Since this slope grows unbounded near the bifurcations, so does the condition number.
Therefore, by approaching such a bifurcation along any of the branches seen in \Cref{fig:joined_eigval_bif_plot}, we can make the condition number arbitrarily large.
Close to these turning points, the eigenvalues exhibit a square-root behavior when viewed as functions of $\delta$. 
That is, if the bifurcation occurs at $\delta_0$, the eigenvalues essentially behave as
\begin{equation}
    |\lambda(\delta) - \lambda(\delta_0)| = c \sqrt{|\delta-\delta_0|} + o(\sqrt{|\delta-\delta_0|}),
\end{equation}
around this point, with $c$ a nonzero constant.
This is a behavior that never occurs for linear symmetric problems, further reiterating that \eqref{eq:NEPv_scale_inv} exhibits characteristics that are qualitatively different to a linear symmetric formulation, and that many things that one would take for granted in the linear setting certainly do not apply for \eqref{eq:NEPv_scale_inv}.
The ill-conditioning is also visualized in the right plot of \Cref{fig:joined_eigval_bif_plot} where the condition number \eqref{eq:condition_number} for each branch is plotted as a function of $\delta$. Indeed, around the turning points $\delta = -3.9$ and $\delta = -1.35$ where two eigenvalues collide, their corresponding condition number approaches infinity. There is an additional singularity at $\delta = -1.73$ for one of the branches because the relative condition number is computed and the corresponding eigenvalue equals zero in this point.

\section{Conclusions and outlook}\label{sec:conc_and_out}

In this paper we considered condition numbers and backward errors for nonlinear eigenvalue problems featuring eigenvector nonlinearities.
We derived closed-form expressions for both of these quantities, both when considering arbitrary perturbations, as well as symmetric ones, applied to the original problem. 
The formulas for the backward errors are of a similar character to analogous formulas for, e.g., the polynomial eigenvalue problem \cite{tisseur2000backward}. 
% while the conditioning formulas differs in important ways from the linear and eigenvalue-nonlinear eigenvalue problems.
For the eigenvalue and eigenvector conditioning, the Jacobian $J(v)$ and its left eigenvector play an important role.
As this Jacobian is usually not symmetric, eigenvalues of \eqref{eq:NEPv_scale_inv} may become severely ill-conditioned even though the original problem is symmetric, which contrasts with linear eigenvalue problems, where symmetric problems have perfectly conditioned eigenvalues.
In \Cref{sec:sims}, we demonstrated that such effects can be encountered even for seemingly innocuous problems by means of two numerical illustrations.
Additionally, our second example, in \Cref{sim:eigval_illcond}, highlights that \eqref{eq:NEPv_scale_inv} can exhibit eigenvalue splitting that can appear highly unintuitive at first glance.
Specifically, we demonstrated that even though we are considering real symmetric problems, eigenvalues can have branch point singularities when considered as functions of a perturbation.
This implies that an eigenvalue can become arbitrarily ill-conditioned by changing a perturbation parameter in the problem.

While the backward error formulas we derived are similar in character to those of other problem classes, they none the less represent an important step towards the implementation of robust numerical solvers for \eqref{eq:NEPv_scale_inv}.
For other types of eigenvalue problems (and many other classes of problems, see references in \Cref{sec:intro}), the backward error provides a natural stopping condition in, e.g., iterative solvers.
Typically, stopping conditions are based on measuring the (relative) residual norm.
Our results essentially show that if the goal is to terminate an iterative procedure when a small backward error has been achieved, this is a valid strategy, so long as the nonlinearities in \eqref{eq:NEPv_scale_inv} are well behaved.
We believe that backward error computations, and by extension the formulas we have derived in this work, will play an important part in the implementation of high-quality software for \eqref{eq:NEPv_scale_inv}, as has been the case historically for many other problems in numerical linear algebra.

The most natural generalization of the results we have presented in this article is perhaps to consider general complex formulations of the NEPv.
We expect that results similar to those we have derived here might also hold for this case. 
However, it is not clear how one would integrate this setting with, e.g., the scaling invariance property \eqref{eq:scale_inv_prop}, which is crucial for our results.
Hence, it is likely that different techniques would be necessary for such a general formulation.
%We expect that most of, if not all, of our results will transfer to this setting, however it is possible that additional complications would arise.
We have also restricted this work to the real symmetric setting 
%for clarity of presentation, and additionally
for the reason that this formulation emphasizes the differences in relation to the linear symmetric eigenvalue problem in a particularly explicit fashion.
Furthermore, many of the important applications of eigenvector nonlinear problems appear in the form \eqref{eq:NEPv_scale_inv}.
A different but equally natural generalization of this work would be to consider the case where we also incorporate perturbations into the nonlinear functions $f_i(v)$.
This would yield a broader analysis of the conditioning, and could be incorporated naturally into different types of nonlinearities, for instance those studied in \cite{janssens2025linearizing}.

Furthermore, we envision that the current work will be viewed as a stepping stone towards a broader perturbation analysis of eigenvector-nonlinear problems.
For instance, the important concept of pseudospectra \cite{trefethen2005pseudospectra} is something that has been very effective for gaining insight into stability aspects associated with eigenvalue-nonlinear problems \cite{michiels2006pseudo}, among other applications.
To the best of the knowledge of the authors, this is something that is yet to be investigated for eigenvector nonlinearities.
A different point of interest is the computational aspects of evaluating, e.g., the condition number.
Condition number estimation has been extensively studied for, among many other problems, linear systems. 
See for instance \cite{higham20001normest}, and the references therein, which also provides an application to the computation of $1$-norm pseudospectra.
The cheap estimation of conditioning (and potentially pseudospectra) could be an important aspect of monitoring the convergence of iterative methods, and would hence be an interesting line of research to pursue in future work.

\section*{Acknowledgements} 
Part of this work was conducted during mutual visits of the first author to KU Leuven and the second author to KTH. 
The first author acknowledges the financial support provided by the foundations managed by the Royal Swedish Academy of Sciences, under grant no. MA2025-0049.  The other authors were supported by the project G027624N of the Research Foundation-Flanders (FWO
- Vlaanderen).

% not included in siam pubs
%\section*{Declarations} \textbf{Conflicts of interest} The authors declare no conflict of interest.

\bibliographystyle{siamplain}
\bibliography{erranalysis}

@article{jarlebring2022implicit,
author = {Jarlebring, Elias and Upadhyaya, Parikshit},
title = {Implicit algorithms for eigenvector nonlinearities},
year = {2022},
publisher = {Springer},
volume = {90},
number = {1},
journal = {Numerical Algorithms},
pages = {301–321},
numpages = {21},
}

@article{tisseur2000backward,
  title={Backward error and condition of polynomial eigenvalue problems},
  author={Tisseur, Françoise},
  journal={Linear Algebra and its Applications},
  volume={309},
  number={1--3},
  pages={339--361},
  year={2000},
  publisher={Elsevier}
}

@article{higham1998structured,
  title={Structured backward error and condition of generalized eigenvalue problems},
  author={Higham, Desmond J. and Higham, Nicholas J.},
  journal={SIAM Journal on Matrix Analysis and Applications},
  volume={20},
  number={2},
  pages={493--512},
  year={1998},
  publisher={SIAM}
}

@article{higham92linsys,
    author = {Higham, Desmond J. and Higham, Nicholas J.},
    title = {Backward error and condition of structured linear systems},
    year = {1992},
    issue_date = {Jan. 1992},
    publisher = {Society for Industrial and Applied Mathematics},
    address = {USA},
    volume = {13},
    number = {1},
    journal = {SIAM Journal on Matrix Analysis and Applications},
    pages = {162–175},
    numpages = {14},
}

@article{ming1998backwarderrbounds,
    author = {Gu, Ming},
    title = {{Backward Perturbation Bounds for Linear Least Squares Problems}},
    journal = {SIAM Journal on Matrix Analysis and Applications},
    volume = {20},
    number = {2},
    pages = {363-372},
    year = {1998},
}

@article{sun1997leastsqbackerr,
    year = {1997},
    author = {Sun, Ji-Guang},
    issn = {0006-3835},
    journal = {BIT Numerical Mathematics},
    language = {eng},
    number = {1},
    pages = {179-188},
    title = {On optimal backward perturbation bounds for the linear least squares problem},
    volume = {37},
}

@Book{Saad:2011:EIGBOOK,
  author	= {Y. Saad},
  title		= {Numerical methods for large eigenvalue problems},
  id		= {1990},
  year		= {2011},
  publisher	= {SIAM},
  zbl_link	= {http://www.zentralblatt-math.org/zmath/en/advanced/?q=an:"1242.65068"}
}

@article{jarlebring2014inverseiter,
    year = {2014},
    author = {Jarlebring, Elias and Kvaal, Simen and Michiels, Wim},
    journal = {SIAM Journal on Scientific Computing},
    number = {4},
    pages = {A1978-A2001},
    title = {{An Inverse Iteration Method for Eigenvalue Problems with Eigenvector Nonlinearities}},
    volume = {36},
}

@article{cai2020perturbation,
  title={Perturbation analysis of an eigenvector-dependent nonlinear eigenvalue problem with applications},
  author={Cai, Yunfeng and Jia, Zhigang and Bai, Zheng-Jian},
  journal={BIT Numerical Mathematics},
  volume={60},
  number={1},
  pages={1--29},
  year={2020},
  publisher={Springer}
}

@book{anderson1999lapackusersguide,
    author = {Anderson, E. and Bai, Z. and Bischof, C. and Blackford, L. S.  and Demmel, J. and Dongarra, J. and Du Croz, J. and Greenbaum, A. and Hammarling, S. and McKenney, A. and Sorensen, D.},
    title = {LAPACK Users' Guide},
    publisher = {Society for Industrial and Applied Mathematics},
    year = {1999},
    edition   = {Third},
}

@TechReport{slepcusersmanual2025,
   author      = "J. E. Roman and C. Campos and L. Dalcin and E. Romero and A. Tomas",
   title       = "{SLEPc} Users Manual",
   number      = "DSIC-II/24/02 - Revision 3.24",
   institution = "D. Sistemes Inform\`atics i Computaci\'o, Universitat Polit\`ecnica de Val\`encia",
   year        = "2025"
}

@Book{Higham:2002:ACCURACY,
  author	= {N. J. Higham},
  title		= {Accuracy and stability of numerical algorithms. 2nd ed},
  id		= {824},
  year		= {2002},
  publisher	= "{SIAM Philadelphia}",
  zbl_link	= {http://www.zentralblatt-math.org/zmath/en/advanced/?q=an:"1011.65010"}
}

@article{janssens2025linearizing,
    title={Linearizing a nonlinear eigenvalue problem with
    quadratic rational eigenvector nonlinearities},
    author={Janssens, Victor and Meerbergen, Karl and Michiels, Wim},
    year={2025},
    note={arXiv preprint},
    institution={arXiv},
    number={arXiv:2510.02900},
}

@Book{Wilkinson:1988:ALGEBRAIC,
  author	= {J. H. Wilkinson},
  title = {The Algebraic Eigenvalue Problem},
  year		= {1988},
  publisher	= {Oxford University Press},
}

@Book{Golub:2013:MATRIX,
  author	= {G. Golub and C. Van Loan},
  title		= {Matrix Computations},
  id		= {2005},
  journal	= {Johns Hopkins Studies in the Mathematical Sciences.
		  Baltimore, MD: The Johns Hopkins University Press. xxi, },
  year		= {2013},
  note		= {4th edition},
  publisher	= {The Johns Hopkins University Press},
  zbl_link	= {http://www.zentralblatt-math.org/zmath/en/advanced/?q=an:"1268.65037"}
}

@article{fraysse1998gepbackerr,
  author       = {Val{\'{e}}rie Frayss{\'{e}} and
                  Vincent Toumazou},
  title        = {A note on the normwise perturbation theory for the regular generalized eigenproblem},
  journal      = {Numerical Linear Algebra with Applications},
  volume       = {5},
  number       = {1},
  pages        = {1--10},
  year         = {1998},
}

@Article{	  Guttel2017,
  author	= {Güttel, S. and Tisseur, F.},
  title		= {The nonlinear eigenvalue problem},
  journal	= {Acta Numerica},
  year		= {2017},
  volume	= {26},
  pages		= {1-94},
  publisher	= {Cambridge University Press}
}

@book{rudin1986mathematical_analysis,
    author = {Rudin, Walter},
    address = {Auckland},
    edition = {3rd},
    publisher = {McGraw-Hill},
    series = {International series in pure and applied mathematics},
    title = {Principles of mathematical analysis},
    year = {1986},
}

@Article{Henning:2025:REVIEW,
  author	= {Patrick Henning and Elias Jarlebring},
  title = {The {Gross-Pitaevskii} equation and eigenvector nonlinearities: Numerical methods and algorithms},
  year = 2025,
 journal = {SIAM Review}
}

@article{bao2003numericalgpe,
    year = {2003},
    author = {Bao, Weizhu and Jaksch, Dieter and Markowich, Peter A.},
    journal = {Journal of Computational Physics},
    number = {1},
    pages = {318-342},
    publisher = {Elsevier Inc},
    title = {{Numerical solution of the Gross–Pitaevskii equation for Bose–Einstein condensation}},
    volume = {187},
}

@article{bao2004normalizedgradflow,
    year = {2004},
    author = {Bao, Weizhu and Du, Qiang},
    journal = {SIAM Journal on Scientific Computing},
    number = {5},
    pages = {1674-1697},
    publisher = {Society for Industrial and Applied Mathematics},
    title = {{Computing the Ground State Solution of Bose--Einstein Condensates by a Normalized Gradient Flow}},
    volume = {25},
}

@article{henning2020sobolevgradflow,
    author = {Henning, Patrick and Peterseim, Daniel},
    title = {{Sobolev Gradient Flow for the Gross--Pitaevskii Eigenvalue Problem: Global Convergence and Computational Efficiency}},
    journal = {SIAM Journal on Numerical Analysis},
    volume = {58},
    number = {3},
    pages = {1744-1772},
    year = {2020},
}

@article {henning2022invitconv,
    AUTHOR = {Henning, Patrick},
     TITLE = {The dependency of spectral gaps on the convergence of the
              inverse iteration for a nonlinear eigenvector problem},
   JOURNAL = {Math. Models Methods Appl. Sci.},
  FJOURNAL = {Mathematical Models and Methods in Applied Sciences},
    VOLUME = {33},
      YEAR = {2023},
    NUMBER = {7},
     PAGES = {1517--1544},
      ISSN = {0218-2025,1793-6314},
   MRCLASS = {35Q55 (65N12 65N25 65N30 81Q80)},
  MRNUMBER = {4595654},
}

@article{yang2009scfanalysis,
    year = {2009},
    author = {Yang, Chao and Gao, Weiguo and Meza, Juan C.},
    journal = {SIAM Journal on Matrix Analysis and Applications},
    number = {4},
    pages = {1773-1788},
    publisher = {Society for Industrial and Applied Mathematics},
    title = {{On the Convergence of the Self-Consistent Field Iteration for a Class of Nonlinear Eigenvalue Problems}},
    volume = {30},
}

@article{upadhyaya2021scfconv,
    title = {A density matrix approach to the convergence of the self-consistent field iteration},
    journal = {Numerical Algebra, Control and Optimization},
    volume = {11},
    number = {1},
    pages = {99-115},
    year = {2021},
    author = {Parikshit Upadhyaya and Elias Jarlebring and Emanuel H. Rubensson},
}

@article{bai2022sharpscfestimate,
    author = {Bai, Zhaojun and Li, Ren-Cang and Lu, Ding},
    title = {{Sharp Estimation of Convergence Rate for Self-Consistent Field Iteration to Solve Eigenvector-Dependent Nonlinear Eigenvalue Problems}},
    year = {2022},
    volume = {43},
    number = {1},
    journal = {SIAM Journal on Matrix Analysis and Applications},
    pages = {301–327},
}

@article{lu2024locallyscf,
  title={Locally unitarily invariantizable {NEPv} and convergence analysis of {SCF}},
  author={Lu, Ding and Li, Ren-Cang},
  journal={Mathematics of Computation},
  volume={93},
  number={349},
  pages={2291--2329},
  year={2024},
}

@article{pulay1982diisacceleration,
    author = {Pulay, P.},
    title = {{Improved SCF convergence acceleration}},
    journal = {Journal of Computational Chemistry},
    volume = {3},
    number = {4},
    pages = {556-560},
    year = {1982}
}

@article{saad2025accelerationfixedpoint,
    title={Acceleration methods for fixed-point iterations},
    volume={34},
    journal={Acta Numerica},
    author={Saad, Yousef},
    year={2025},
    pages={805–890}
}

@article{jarlebring2026nepvtonep_BIT,
  title={From eigenvector nonlinearities with quadratic structure to eigenvalue nonlinearities with algebraic structure},
  author={Jarlebring, Elias and Lithell, Vilhelm P},
  journal={BIT Numerical Mathematics},
  volume={66},
  number={2},
  pages={30},
  year={2026},
  publisher={Springer}
}

@article{claes2022linearizenepv,
    year = {2022},
    author = {Claes, Rob and Jarlebring, Elias and Meerbergen, Karl and Upadhyaya, Parikshit},
    journal = {SIAM Journal on Matrix Analysis and Applications},
    language = {eng},
    number = {2},
    pages = {764-786},
    title = {{Linearizable Eigenvector Nonlinearities}},
    volume = {43},
}

@article{claes2023contour,
    year = {2023},
    author = {Claes, Rob and Meerbergen, Karl and Telen, Simon},
    journal = {SIAM Journal on Matrix Analysis and Applications},
    number = {4},
    pages = {1619-1644},
    title = {{Contour Integration for Eigenvector Nonlinearities}},
    volume = {44},
}

@book{martin2020electronicstructurecalc,
    place={Cambridge},
    edition={2},
    title={{Electronic Structure: Basic Theory and Practical Methods}},
    publisher={Cambridge University Press},
    author={Martin, Richard M.},
    year={2020}
}

@article{truhar2021relativepert,
    title = {On an eigenvector-dependent nonlinear eigenvalue problem from the perspective of relative perturbation theory},
    journal = {Journal of Computational and Applied Mathematics},
    volume = {395},
    pages = {113596},
    year = {2021},
    author = {Ninoslav Truhar and Ren-Cang Li},
}

@InProceedings{hein2010pspectral,
  author        = {Matthias Hein and Thomas B\"uhler},
  title         = {{An Inverse Power Method for Nonlinear Eigenproblems with
                  Applications in 1-Spectral Clustering and Sparse PCA}},
  booktitle     = {Advances in Neural Information Processing Systems 23 (NIPS
                  2010)},
  year          = {2010},
}

@InProceedings{hein2009pspectral,
  author        = {T. B\"{u}hler and M. Hein},
  title         = {Spectral {C}lustering Based on the {G}raph {p-Laplacian}},
  booktitle     = {Proceedings of the 26th International Conference on
                  Machine Learning},
  pages         = {81-88},
  year          = {2009}
}

@article{bai2018robust,
    author = {Bai, Zhaojun and Lu, Ding and Vandereycken, Bart},
    title = {{Robust Rayleigh Quotient Minimization and Nonlinear Eigenvalue Problems}},
    journal = {SIAM Journal on Scientific Computing},
    volume = {40},
    number = {5},
    pages = {A3495-A3522},
    year = {2018},
}

@article{hochtenbach2008harmritz,
    author = {Hochstenbach, Michiel E. and Sleijpen, Gerard L. G.},
    title = {{Harmonic and refined Rayleigh–Ritz for the polynomial eigenvalue problem}},
    journal = {Numerical Linear Algebra with Applications},
    volume = {15},
    number = {1},
    pages = {35-54},
    year = {2008}
}

@article{paige1995harmritz,
    author = {Paige, Chris C. and Parlett, Beresford N. and van der Vorst, Henk A.},
    title = {Approximate solutions and eigenvalue bounds from {K}rylov subspaces},
    journal = {Numerical Linear Algebra with Applications},
    volume = {2},
    number = {2},
    pages = {115-133},
    year = {1995}
}

@article{michiels2006pseudo,
    title = {Pseudospectra and stability radii for analytic matrix functions with application to time-delay systems},
    journal = {Linear Algebra and its Applications},
    volume = {418},
    number = {1},
    pages = {315-335},
    year = {2006},
    author = {Wim Michiels and Kirk Green and Thomas Wagenknecht and Silviu-Iulian Niculescu},
}

@article{higham20001normest,
    year = {2000},
    author = {Higham, Nicholas J. and Tisseur, Françoise},
    journal = {SIAM Journal on Matrix Analysis and Applications},
    number = {4},
    pages = {1185-1201},
    publisher = {Society for Industrial and Applied Mathematics},
    title = {A Block Algorithm for Matrix 1-Norm Estimation, with an Application to 1-Norm Pseudospectra},
    volume = {21},
}

@book{trefethen2005pseudospectra,
  author= {L. N. Trefethen and M. Embree},
  title = {Spectra and Pseudospectra. The Behavior of Nonnormal
		  Matrices and Operators},
  id	= {753},
  year = {2005},
  publisher = {Princeton University Press},
}
\end{document}